\documentclass{amsart}

\usepackage{amssymb,mathrsfs,eucal}
\usepackage{graphicx}
\usepackage{pstricks}
\psset{unit=0.5} 

\newtheorem{theorem}{Theorem}[section]
\newtheorem{lemma}[theorem]{Lemma}
\newtheorem{cor}[theorem]{Corollary}
\newtheorem{prop}[theorem]{Proposition}

\theoremstyle{definition}
\newtheorem{definition}[theorem]{Definition}
\newtheorem{example}[theorem]{Example}
\newtheorem*{aspta}{Assumption A}

\theoremstyle{remark}
\newtheorem{remark}[theorem]{Remark}

\numberwithin{equation}{section}

\newcommand{\one}{\mathbf{1}}

\newcommand{\vvverto}{\mathopen{|\hspace{-0.12em}|\hspace{-0.12em}|}}
\newcommand{\vvvertc}{\mathclose{|\hspace{-0.12em}|\hspace{-0.12em}|}}
\newcommand{\tnorm}[1]{\vvverto #1 \vvvertc}
\newcommand{\ii}{\mathrm{i}}

\begin{document}

\title{The local geometry of finite mixtures}


\author{Elisabeth Gassiat}
\address{Laboratoire de Math{\'e}matiques, 
Universit{\'e} Paris-Sud, 
B{\^a}timent 425, 
91405 Orsay Cedex, France}
\email{elisabeth.gassiat@math.u-psud.fr}

\author{Ramon van Handel}
\address{Sherrerd Hall, Room 227,
Princeton University, 
Princeton, NJ 08544, USA.}
\email{rvan@princeton.edu}



\date{}


\begin{abstract}
We establish that for $q\ge 1$, the class of convex combinations of $q$ 
translates of a smooth probability density has local doubling dimension 
proportional to $q$.  The key difficulty in the proof is to control the
local geometric structure of mixture classes.  Our local geometry theorem
yields a bound on the (bracketing) metric entropy of a class of normalized
densities, from which a local entropy bound is deduced by a
general slicing procedure.
\end{abstract}

\maketitle

\section{Introduction}
\label{sec:intro}

Let $(X,d)$ be a metric space, and consider a subset 
$T=\{t_\xi:\xi\in\Xi\}$ of $X$ that is parametrized by a bounded subset 
$\Xi$ of $\mathbb{R}^d$.  Roughly speaking, we are interested in the 
following question: can $T$ be viewed as a \emph{finite-dimensional} 
subset of $X$?  It is certainly tempting to think so, as the parameter 
set $\Xi$ is finite-dimensional. This idea is easily made precise if the 
induced metric $d_T(\xi,\xi')=d(t_\xi,t_{\xi'})$ on $\Xi$ is comparable 
to a norm on $\mathbb{R}^d$, so that $T$ inherits the Euclidean 
geometry.  However, there are natural examples whose geometry is highly 
non-Euclidean, so that the conclusion is far from obvious. The aim of 
this paper is to investigate in detail such a problem that arises from 
applications in statistics.

To set the stage for the problem that we will consider, let us recall 
some metric notions of dimension.  For a subset $T$ of
a metric space $(X,d)$, the \emph{covering number} $N(T,\varepsilon)$ is 
the smallest cardinality of a covering of $T$ by
$\varepsilon$-balls \cite{KT59}:
$$
	N(T,\varepsilon) = \inf\bigg\{
	n:\exists~x_i\in X,~i=1,\ldots,n\mbox{ s.t.\ }
	T\subseteq\bigcup_{i=1}^n B(x_i,\varepsilon)
	\bigg\},
$$
where $B(x,\varepsilon)=\{x'\in X:d(x,x')\le\varepsilon\}$.  The covering
number, or equivalently the metric entropy
$\log N(T,\varepsilon)$, quantifies the capacity of the set $T$,
and its scaling in $\varepsilon$ is closely connected to
dimension. Indeed, let $|\cdot|$ be a norm on $\mathbb{R}^d$, so that 
$(\mathbb{R}^d,|\cdot|)$ is a finite-dimensional Banach space.  A standard
estimate \cite[Lemma 4.14]{Mas07} gives
$$
	N(B(t,\delta),\varepsilon) \le \bigg(
	\frac{3\delta}{\varepsilon}
	\bigg)^d
$$
for any $\varepsilon\le\delta$, where
$B(t,\delta)=\{x\in\mathbb{R}^d:|x-t|\le\delta\}$.
This estimate has two trivial
consequences: first, for any bounded $T\subset 
(\mathbb{R}^d,|\cdot|)$, there is a constant $C_1$ so that
\begin{equation}
\label{eq:globalent}
	N(T,\varepsilon) \le \bigg(\frac{C_1}{\varepsilon}\bigg)^d
\end{equation}
for all $\varepsilon$ sufficiently small.  On the other hand, if we fix
a distinguished point $t_0\in T$, there is a constant $C_2$ such that
for all $\varepsilon/\delta$ sufficiently small
\begin{equation}
\label{eq:localent}
	N(T\cap B(t_0,\delta),\varepsilon) \le
	\bigg(\frac{C_2\delta}{\varepsilon}\bigg)^d.
\end{equation}
Either (\ref{eq:globalent}) or (\ref{eq:localent}) may be used as a
notion of finite-dimensionality for a set $T$ in a general metric space 
$(X,d)$: a set satisfying the \emph{global} entropy
bound (\ref{eq:globalent}) has finite Kolmogorov dimension
$\log N(T,\varepsilon)/\log(1/\varepsilon)\lesssim d$, while a
set satisfying the \emph{local} entropy bound (\ref{eq:localent}) has finite
local\footnote{
	The doubling (Assouad) dimension of a set $T$ is defined as
	the supremum of the local doubling dimension 
	$\sup_{\varepsilon}\log N(T\cap B(t_0,2\varepsilon),\varepsilon)$
	with respect to $t_0$
	\cite{Ass83,Hei01}.  For the purposes of this paper, we will
	consider mainly the local version of this concept where
	the point $t_0$ is fixed.
} doubling dimension 
$\log N(T\cap B(t_0,2\varepsilon),\varepsilon)
\lesssim d$.
Clearly (\ref{eq:localent}) implies
(\ref{eq:globalent}), but not conversely.

Now consider a parametrized set $T=\{t_\xi:\xi\in\Xi\}$ in a metric 
space $(X,d)$, where $\Xi$ is a bounded subset of $\mathbb{R}^d$, and
let $|\cdot|$ be a norm on $\mathbb{R}^d$.
As $(\Xi,|\cdot|)$ is finite-dimensional in either sense (\ref{eq:globalent})
or (\ref{eq:localent}), these properties are inherited by $T$ provided that
the metric $d$ is comparable to $|\cdot|$.
Indeed, if we have a H\"older-type upper bound
$d(t_\xi,t_{\xi'})\le C|\xi-\xi'|^\alpha$, then $T$ satisfies the
global entropy bound (\ref{eq:globalent}); if we have in addition the
lower bound $d(t_\xi,t_{\xi_0})\ge c|\xi-\xi_0|^\alpha$, we obtain
the local entropy bound (\ref{eq:localent}) with $t_0=t_{\xi_0}$.\footnote{
	If $d(t_\xi,t_{\xi'})\le C|\xi-\xi'|^\alpha$,
	then any covering of $\Xi$ by balls of radius 
	$(\varepsilon/C)^{1/\alpha}$ yields a covering 
	of $T$ by $\varepsilon$-balls, so that
	$N(T,\varepsilon)\le N(\Xi,(\varepsilon/C)^{1/\alpha}) \le 
	(C'/\varepsilon)^{d/\alpha}$.  If also
	$d(t_\xi,t_{\xi_0})\ge c|\xi-\xi_0|^\alpha$, then
	$\{\xi\in\Xi:d(t_\xi,t_{\xi_0})\le\delta\}\subseteq
	\Xi\cap B(\xi_0,(\delta/c)^{1/\alpha})$, so
	$N(T\cap B(t_{\xi_0},\delta),\varepsilon)\le
	(C''\delta/\varepsilon)^{d/\alpha}$.
}
The upper bound is easily obtained in many 
cases of interest, so that finite-dimensionality in the sense
(\ref{eq:globalent}) is not too problematic.
The lower bound is much more delicate, however.
In its absence, finite-dimensionality in the sense (\ref{eq:localent})
is far from obvious.

We will investigate these issues in the context of a prototypical 
example, to be described presently, that is of significant independent 
interest. Fix a probability density $f_0$ on $\mathbb{R}^d$ 
(that is, $f_0\ge 0$ and $\int f_0\,dx=1$), and consider the class
$$
	\mathcal{M}_q = \bigg\{ x\mapsto
	\sum_{i=1}^q\pi_if_0(x-\theta_i):
	\pi_i\ge 0,~
	\sum_{i=1}^q \pi_i = 1,~
	\theta_i \in \Theta
	\bigg\}
$$
of convex combinations of $q$ translates of $f_0$, where $\Theta$ is a 
bounded subset of $\mathbb{R}^d$. Such densities appear in numerous 
statistical applications, where they are frequently known as \emph{location 
mixtures}. $\mathcal{M}_q$ is a subset of the space $\mathcal{M}$ of all 
probability densities on $\mathbb{R}^d$, endowed with a suitable metric 
$d$.  

$\mathcal{M}_q$ is parametrized by the finite-dimensional subset 
$\Xi_q=\Delta_{q-1}\times\Theta^q$ of $\mathbb{R}^{qd+q-1}$, where 
$\Delta_{q-1}$ is the $q$-simplex. Natural metrics $d$ satisfy a 
H\"older-type upper bound with respect to a norm on $\Xi_q$ (e.g., step 
2 in the proof of Theorem \ref{thm:mainmixtures} below). However, the 
corresponding lower bound is impossible to obtain.

\begin{example}
\label{ex:geom}
We will write $f_\theta(x)= f_0(x-\theta)$ for simplicity.
Fix $\theta^\star\in\Theta$ and let $f^\star=f_{\theta^\star}$.
Then $f^\star\in\mathcal{M}_2$, but $f^\star$ is not uniquely
represented by a parameter in $\Xi_2$:
\begin{align*}
	&\{(\pi,\theta)\in\Xi_2:
	d(\pi_1f_{\theta_1}+\pi_2f_{\theta_2},f^\star)=0\}
	={}\\
	&\{\pi\in\Delta_1,\theta_1=\theta_2=\theta^\star\}\cup
	\{\pi_1=0,\theta_1\in\Theta,\theta_2=\theta^\star\}\cup
	\{\pi_1=1,\theta_1=\theta^\star,\theta_2\in\Theta\}.
\end{align*}
Clearly $d$ cannot be lower bounded by any norm on $\Xi_2$,
as such a bound would necessarily imply that
$\{(\pi,\theta)\in\Xi_2:
        d(\pi_1f_{\theta_1}+\pi_2f_{\theta_2},f^\star)=0\}$
consists of a single point.  Thus the above approach to 
(\ref{eq:localent}) is useless here.
\end{example}

\begin{figure}
\centering
\includegraphics[width=.85\textwidth]{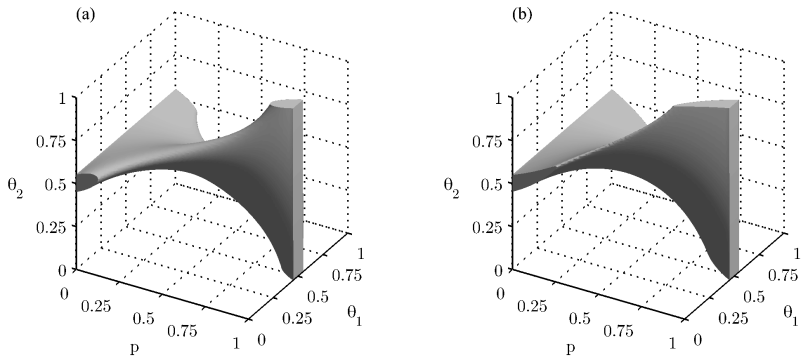}
\vskip-.4cm
\caption{\small 
Let $f_\theta(x)=e^{-2(x-\theta)^2}$, $f^\star = f_{0.5}$,
$\mathcal{M}_2=\{pf_{\theta_1}+
(1-p)f_{\theta_2}:p,\theta_1,\theta_2\in[0,1]\}$.
The plots illustrate
(a) the set of parameters $(p,\theta_1,\theta_2)$ corresponding
to the Hellinger ball $\{f\in\mathcal{M}_2:h(f,f^\star)\le 0.05\}$; and
(b) the parameter set 
$\{(p,\theta_1,\theta_2):N(p,\theta_1,\theta_2)\le 0.05\}$ with
$N(p,\theta_1,\theta_2) = 
|p(\theta_1-0.5)+(1-p)(\theta_2-0.5)|
+\tfrac{1}{2}p(\theta_1-0.5)^2
+\tfrac{1}{2}(1-p)(\theta_2-0.5)^2$.  The two plots are related by
the local geometry 
Theorem \ref{thm:localgeom}, which yields
$c^\star N(p,\theta_1,\theta_2)\le h(pf_{\theta_1}+(1-p)f_{\theta_2},f^\star)
\le C^\star N(p,\theta_1,\theta_2)$.
}
\label{fig:hball}
\end{figure}

The phenomenon illustrated in this example can be stated more generally. 
For $f^\star\in\mathcal{M}_{q^\star}$ such that $q^\star<q$ (note that 
$f^\star\in\mathcal{M}_q$ as $\mathcal{M}_{q}\subset\mathcal{M}_{q+1}$ 
for all $q$), the subset of parameters $\Xi_q(\delta)\subset\Xi_q$ 
corresponding to the ball 
$\mathcal{M}_q(\delta)=\{f\in\mathcal{M}_q:d(f,f^\star)\le\delta\}$ 
behaves nothing at all like a ball in a finite-dimensional Banach space 
(see Figure \ref{fig:hball}(a)): indeed, the diameter of $\Xi_q(\delta)$ 
is even bounded away from zero as $\delta\downarrow 0$. There is 
therefore no hope to deduce a local entropy bound of the form 
(\ref{eq:localent}) for $N(\mathcal{M}_q(\delta),\varepsilon)$ directly 
from the corresponding bound in $\mathbb{R}^{qd+q-1}\supset\Xi_q$.  
This provides a vivid illustration of the difficulty of 
establishing local entropy bounds in geometrically irregular settings.  
Nevertheless, we will be able to obtain local entropy bounds for the 
mixture classes $\mathcal{M}_q$ in section \ref{sec:mixtures} below.

For concreteness, we endow $\mathcal{M}_q$ with the 
Hellinger metric $h(f,g)=\|\sqrt{f}-\sqrt{g}\|_{L^2}$, which is the 
relevant metric for statistical applications \cite[ch.\ 7]{vdG00}, 
\cite{Mas07} (however, our results are easily adapted to other commonly 
used probability metrics---the total variation metric $d_{\rm 
TV}(f,g)=\|f-g\|_{L^1}$, for example---using almost identical proofs). 
The main result, Theorem \ref{thm:localmixtures}, provides an explicit 
bound of the form (\ref{eq:localent}) for $\mathcal{M}_q$ under suitable 
smoothness assumptions on $f_0$.

The fundamental challenge that we face in the proof is to develop a 
sharp quantitative understanding of the local geometry of mixtures 
(illustrated in Figure \ref{fig:hball}). The key result that we prove in 
this direction is Theorem \ref{thm:localgeom}, which forms the central
contribution of this paper.  As this result is rather technical,
we postpone its description to section \ref{sec:mixturesGeo} below.
However, an important consequence of this result is as follows:
given a mixture $f^\star = 
\sum_{i=1}^{q^\star}\pi_i^\star f_{\theta_i^\star}$, one can choose 
sufficiently small neighborhoods $A_1,\ldots,A_{q^\star}$ of 
$\theta_1,\ldots,\theta_{q^\star}$, respectively, such that for any
$q\ge 1$ and 
mixture $f=\sum_{i=1}^q\pi_if_{\theta_i}$, the Hellinger metric
$h(f,f^\star)$ is of the same order as
$$
	\sum_{\theta_j\in A_0}\pi_j +\sum_{i=1}^{q^\star}
        \Bigg\{
        \Bigg|
        \sum_{\theta_j\in A_i}\pi_j
        -\pi_i^\star\Bigg| 
        \\ \mbox{}
        +
        \Bigg\|
        \sum_{\theta_j\in A_i}\pi_j(\theta_j-\theta_i^\star)
        \Bigg\|  
        + \frac{1}{2}
        \sum_{\theta_j\in A_i}\pi_j\|\theta_j-\theta_i^\star\|^2
        \Bigg\}
$$
(here $A_0=\mathbb{R}^d\backslash (A_1\cup\cdots\cup A_{q^\star})$).  
This pseudodistance controls precisely the set of parameters in $\Xi_q$ 
with density close to $f^\star$, see Figure \ref{fig:hball} for an
example. 

Let us emphasize that while the local geometry theorem relates the 
Hellinger metric on $\mathcal{M}_q$ to a pseudodistance on $\Xi_q$, 
the latter is not a norm or even a metric.  It is therefore still not 
possible to control the local entropy of $\mathcal{M}_q$ as in the case 
where the metric is comparable to a norm on $\Xi_q$.  Instead, we deduce 
the local entropy bound in two steps.  First, we observe that the local 
geometry theorem allows us to obtain a \emph{global} entropy bound of 
the form (\ref{eq:globalent}) for the class of weighted densities
$$
	\mathcal{D}_q=\bigg\{
	\frac{\sqrt{f/f^\star}-1}{h(f,f^\star)}:
	f\in\mathcal{M}_q,~f\ne f^\star
	\bigg\},
$$
as the above pseudodistance controls the coefficients 
in the Taylor expansion of $f$.  This is accomplished in Theorem 
\ref{thm:mainmixtures}.  The global entropy bound for 
$\mathcal{D}_q$ now yields a local entropy bound for $\mathcal{M}_q$
using a slicing procedure.  The latter is not specific to 
mixtures, and will be developed first in a general setting in section 
\ref{sec:entropy}.

Beside their intrinsic interest, the results in this paper are of direct 
relevance to statistical applications.  Many problems in statistics and 
probability make use of estimates on the metric entropy of classes of 
densities: metric entropy controls the rate of convergence of uniform 
limit theorems in probability, and is therefore of central importance in 
the design and analysis of statistical estimators 
\cite{VW96,vdG00,Mas07}.  Such applications frequently require a 
slightly stronger notion of metric entropy known as bracketing entropy, 
which we will consider throughout this paper; see section 
\ref{sec:entropy}. In infinite-dimensional situations, the global 
entropy is chiefly of interest: global entropy estimates for various 
classes of probability densities can be found in 
\cite{VW96,vdG00,Mas07,BGL07,GLW10}.  However, in finite-dimensional 
settings, global entropy bounds are known to yield sub-optimal results, 
and here local entropy bounds are essential to obtain optimal 
convergence rates of estimators \cite[\S 7.5]{vdG00}. In 
the case of mixtures, the difficulty of obtaining local entropy bounds 
was noted, e.g., in \cite{GW00,MM11}. Applications of the results 
in this paper are given in \cite{GasHanOrder,GasRouHMM}.

\section{From global entropy to local entropy}
\label{sec:entropy} 

The classical notion of covering numbers $N(T,\varepsilon)$ was defined 
in the introduction.  We will consider throughout this paper a somewhat 
finer notion of covering by brackets (order intervals) rather than by 
balls.  In this section, we will work in the general setting of normed 
vector lattices (normed Riesz spaces, see \cite{AB06}).

\begin{definition}
Let $(X,\|\cdot\|)$ be a normed vector lattice.  For any subset
$T\subseteq X$ and $\varepsilon>0$, the bracketing number 
$N_{[]}(T,\varepsilon)$ is defined as
$$
	N_{[]}(T,\varepsilon) = \inf\bigg\{
	n:\exists~l_i,u_i\in X,~\|u_i-l_i\|\le\varepsilon,~
	i=1,\ldots,n
	\mbox{ s.t.\ }
	T\subseteq\bigcup_{i=1}^n [l_i,u_i]
	\bigg\},
$$
where $[l,u]=\{x\in X:l\le x\le u\}$.
\end{definition}

Note that as $[l,u]\subset B(l,\|u-l\|)$, it is 
evident that $N(T,\varepsilon)\le N_{[]}(T,\varepsilon)$ for any 
$T\subseteq X$ and $\varepsilon>0$.  Bounds on the bracketing number 
therefore imply bounds on the covering number, but not conversely.
The finer covering by brackets is essential in many
probabilistic and statistical applications \cite{VW96,vdG00,Mas07}.

Let $(X,\|\cdot\|)$ be a normed vector lattice, and let us fix a subset 
$T\subseteq X$ and a distinguished point $t_0\in T$.  Our general aim is to 
obtain an estimate on the local covering (or bracketing) number 
$N(T\cap B(t_0,\delta),\varepsilon)$ that is polynomial in 
$\delta/\varepsilon$.  As is explained in the introduction, such 
estimates can be much more difficult to obtain than the corresponding 
estimates on the global covering number $N(T,\varepsilon)$ that 
are polynomial in $1/\varepsilon$.  Unfortunately, the latter is strictly 
weaker than the former.

Nonetheless, global covering estimates can be useful.  For any $t\ne t_0$,
define
$$
	d_t = \frac{t-t_0}{\|t-t_0\|},\qquad\quad
	D_0 = \{d_t:t\in T,~t\ne t_0\}. 
$$ 
The main message of this section is that a \emph{local} covering 
estimate for $T$ can be obtained from a \emph{global} covering estimate 
for the weighted class $D_0\subseteq X$. As global entropy estimates can 
be much easier to obtain than local entropy estimates, this provides a 
useful approach to obtaining local entropy bounds for geometrically 
complex classes.  We state a precise result for bracketing numbers
as will be needed in the sequel; a trivial modification of the proof
yields a version for covering numbers.  In the next section, this result
will be applied in the context of mixtures.

\begin{theorem}
\label{thm:localglobal}
Let $(X,\|\cdot\|)$ be a normed vector lattice.  Fix $T\subseteq X$
and $t_0\in T$, and let $D_0$ be as above.
Suppose that there exist $q,C_0\ge 1$ and $\varepsilon_0>0$ such that 
$$
        N_{[]}(D_0,\varepsilon)\le
        \bigg(\frac{C_0}{\varepsilon}\bigg)^q
        \quad\mbox{for every }\varepsilon\le\varepsilon_0.
$$
Choose any $d\in X$ such that $|d_t|\le d$ for all $t\in T$, $t\ne t_0$.
Then
$$
        N_{[]}(T\cap B(t_0,\delta),\rho) \le 
        \bigg(
        \frac{8C\delta}{\rho}
        \bigg)^{q+1}
$$
for all $\delta,\rho>0$ such that $\rho/\delta < 4 \wedge 2\|d\|$,
where $C = C_0(1 \vee \|d\|/4\varepsilon_0)$.
\end{theorem}

\begin{remark}
Theorem \ref{thm:localglobal} requires an upper bound $d\in X$ on 
$|D_0|$, that is, $D_0$ must be order-bounded.  But the assumptions
of the Theorem already require that $N_{[]}(D_0,\varepsilon_0)<\infty$,
which is easily seen to imply order-boundedness of $D_0$.  The latter
therefore does not need to be added as a separate assumption.
\end{remark}

\begin{remark}
In Theorem \ref{thm:localglobal}, a global covering bound for $D_0$ of order 
$(1/\varepsilon)^{q}$ gives a local covering bound for $T$ of order 
$(\delta/\varepsilon)^{q+1}$.  It is instructive to note that this 
polynomial scaling cannot be improved.  Indeed, let $T$ be the unit
(Euclidean) ball in $\mathbb{R}^{q+1}$, and let $t_0=0$.
Then $D_0$ is the unit sphere in $\mathbb{R}^{q+1}$ and therefore
has Kolmogorov dimension $q$, but the covering number of 
$B(0,\delta)$ is of order $(\delta/\varepsilon)^{q+1}$.
The same conclusion holds also for bracketing (rather than covering)
numbers.
\end{remark}

\begin{remark}
A natural question is whether a converse to the above results can be 
obtained.  In general, however, this is not possible: the class $D_0$
can be much richer than the original class $T$, as the following simple
example illustrates.  Let $(X,\|\cdot\|)$ be an infinite-dimensional Hilbert
lattice and let $(e_k)_{k\ge 1}$ be an orthonormal basis.  Let
$T=\{2^{-k}e_k:k\ge 1\}\cup\{0\}$ and $t_0=0$.  Then
$N_{[]}(T\cap B(t_0,2^{-r}),2^{-k}) \le k-r+1$ for $k\ge r$, so
$N_{[]}(T\cap B(t_0,\delta),\varepsilon) \le
\log_2(8\delta/\varepsilon) \le (8\delta/\varepsilon)^{3/2}$
for all $\varepsilon/\delta\le 1$.
But here we have $D_0 = \{e_k:k\ge 1\}$, so
$N_{[]}(D_0,\varepsilon)\ge N(D_0,\varepsilon)=\infty$ for $\varepsilon>0$ small enough.
\end{remark}

We now proceed to the proof of Theorem \ref{thm:localglobal}.  The main 
idea of the proof is to partition the set $T\cap B(t_0,\delta)$ into 
shells $\{t\in T:r^{-n}\delta\le\|t-t_0\|\le r^{-n+1}\delta\}$ for a 
suitable choice of $r>0$.  The bracketing number of each shell is then 
controlled by that of the normalized class $D_0$ at scale $\sim 
r^n\rho/\delta$.  Such a slicing procedure is commonly used in the
reverse direction in the theory of weighted empirical processes
(see, e.g., \cite[sec.\ 5.3]{vdG00}).  Here we apply this
idea directly to the bracketing numbers.

\begin{proof}[Proof of Theorem \ref{thm:localglobal}]
The assumption implies that
$$
        N_{[]}(D_0,\varepsilon)\le
        \bigg(\frac{C_0}{\varepsilon\wedge\varepsilon_0}\bigg)^q
        \quad\mbox{for every }\varepsilon>0.
$$
If $\varepsilon<\|d\|/4$, then
$$
        \frac{\varepsilon}{\varepsilon\wedge\varepsilon_0}
        \le 1\vee \frac{\|d\|}{4\varepsilon_0}.
$$
We therefore have
$$
        N_{[]}(D_0,\varepsilon)\le
        \bigg(\frac{C}{\varepsilon}\bigg)^q
        \quad\mbox{for every }\varepsilon<\|d\|/4,
$$
where $C$ is as defined in the Theorem.  This estimate will be used below.

Fix $\varepsilon,\delta>0$ and let 
$N=N_{[]}(D_0,\varepsilon)$.  Then there exist
$l_1,u_1,\ldots,l_N,u_N\in X$ such that $\|u_i-l_i\|\le\varepsilon$
for all $i=1,\ldots,N$, and for every $t\in T$, $t\ne t_0$ there is an 
$1\le i\le N$ such that
$l_i\le d_t\le u_i$.  Choose $t\in T$ such that $r^{-n}\delta\le
\|t-t_0\|\le r^{-n+1}\delta$ (with $r>1$ to be chosen later).  
Then there exists $1\le i\le N$ so that
$$
        (r^{-n}l_i\wedge r^{-n+1}l_i)\,\delta + t_0 \le
        t \le
        (r^{-n}u_i\vee r^{-n+1}u_i)\,\delta + t_0.
$$
Note that
\begin{align*}
        \|u_i\,r^{-n}\delta-l_i\,r^{-n}\delta\| &\le
        r^{-n}\delta\varepsilon,\\
        \|u_i\,r^{-n+1}\delta-l_i\,r^{-n+1}\delta\| &\le
        r^{-n+1}\delta\varepsilon,\\
        \|u_i\,r^{-n+1}\delta-l_i\,r^{-n}\delta\| &\le
        (r-1)r^{-n}\delta + r^{-n+1}\delta\varepsilon,\\
        \|u_i\,r^{-n}\delta-l_i\,r^{-n+1}\delta\| &\le
        (r-1)r^{-n}\delta + r^{-n+1}\delta\varepsilon,
\end{align*}
where the latter two estimates follow from
$l_i\le d_t\le u_i$, $\|d_t\|=1$, and
\begin{align*}
        (u_i-l_i)\,r^{-n}\delta &\le
        u_i\,r^{-n+1}\delta-l_i\,r^{-n}\delta 
        - d_t\,(r-1)r^{-n}\delta
        \le
        (u_i-l_i)\,r^{-n+1}\delta, \\
        (u_i-l_i)\,r^{-n}\delta &\le
        u_i\,r^{-n}\delta-l_i\,r^{-n+1}\delta 
        + d_t\,(r-1)r^{-n}\delta
        \le
        (u_i-l_i)\,r^{-n+1}\delta.
\end{align*}
As $|a\vee b-c\wedge d|\le |a-c|+|a-d|+|b-c|+|b-d|$, we can estimate
$$
        \|(r^{-n}u_i\vee r^{-n+1}u_i)\,\delta-
        (r^{-n}l_i\wedge r^{-n+1}l_i)\,\delta\| \le
        2(r-1)r^{-n}\delta + 4r^{-n+1}\delta\varepsilon.
$$
Therefore, we have shown that
$$
        N_{[]}(\{t\in T:r^{-n}\delta\le
        \|t-t_0\|\le r^{-n+1}\delta\},
        2(r-1)r^{-n}\delta + 4r^{-n+1}\delta\varepsilon) \le 
        N_{[]}(D_0,\varepsilon)
$$
for arbitrary $\varepsilon,\delta>0$, $r>1$, $n\in\mathbb{N}$.
In particular,
$$
        N_{[]}(\{t\in T:r^{-n}\delta\le
        \|t-t_0\|\le r^{-n+1}\delta\},\rho) \le 
        N_{[]}(D_0,\tfrac{1}{4}r^{n-1}\rho/\delta - 
        \tfrac{1}{2}(1-1/r))
$$
for every $\delta>0$, $r>1$, $n\in\mathbb{N}$, 
$\rho>2(r-1)r^{-n}\delta$.

Choose an envelope $d\in X$ such that $|d_t|\le d$ for all $t\in T$,
$t\ne t_0$.  Evidently
$$
        t_0-r^{-n}\delta\, d\le t \le t_0+r^{-n}\delta\,d
$$
for all $t\in T$ such that $\|t-t_0\|\le r^{-n}\delta$.  Therefore
$$
        N_{[]}(\{t\in T:\|t-t_0\|\le r^{-
        \lceil H\rceil}\delta\},
        2r^{-H}\delta\|d\|)=1
$$
for all $\delta>0$, $r>1$, $H>0$.  Thus we can estimate
\begin{align*}
        &N_{[]}(T\cap B(t_0,\delta),
        2r^{-H}\delta\|d\|) \\
        &\quad\le
        1+\sum_{n=1}^{\lceil H\rceil}
        N_{[]}(\{t\in T:r^{-n}\delta\le
        \|t-t_0\|\le r^{-n+1}\delta\},2r^{-H}\delta\|d\|) \\
        &\quad\le 1+\sum_{n=1}^{\lceil H\rceil}
        N_{[]}(D_0,\{r^{n-H-1}\|d\|- 
        (1-1/r)\}/2)
\end{align*}
whenever $\delta>0$, $r>1$, $H>0$ such that $\|d\| > 
(1-1/r)r^{H}$.  In particular,
$$
        N_{[]}(T\cap B(t_0,\delta),
        2r^{-H}\delta\|d\|) 
        \le 1+\sum_{n=1}^{\lceil H\rceil}
        N_{[]}(D_0,r^{n-H-1}\|d\|/4)
$$
whenever $\delta>0$, $r>1$, $H>0$ such that $\|d\| \ge
2(1-1/r)r^{H}$, where we have used that the bracketing number is
a nonincreasing function of the bracket size.

Now recall that 
$$
        N_{[]}(D_0,\varepsilon)\le
        \bigg(\frac{C}{\varepsilon}\bigg)^q
        \quad\mbox{for every }0<\varepsilon<\|d\|/4,
$$
where $q,C\ge 1$.  Thus
$$
        N_{[]}(T\cap B(t_0,\delta),
        2r^{-H}\delta\|d\|) \le 1+
        \sum_{n=1}^{\lceil H\rceil}r^{-(n-1)q}
        \bigg(
        \frac{8C}{2r^{-H}\|d\|}
        \bigg)^q
$$
whenever $\delta>0$, $r>1$, $H>0$ such that $\|d\| \ge
2(1-1/r)r^{H}$.  But
$$
        \sum_{n=1}^{\lceil H\rceil}r^{-(n-1)q} \le
        \frac{1}{1-1/r^q} \le
        \frac{1}{1-1/r} \le
        \frac{\|d\|}{2(1-1/r)r^{H}}
        \frac{4C}{2r^{-H}\|d\|}
$$
as $r>1$ and $q,C\ge 1$.  We can therefore estimate
$$
        N_{[]}(T\cap B(t_0,\delta),
        2r^{-H}\delta\|d\|) \le 
        \frac{\|d\|}{2(1-1/r)r^{H}}
        \bigg(
        \frac{8C}{2r^{-H}\|d\|}
        \bigg)^{q+1}
$$
whenever $\delta>0$, $r>1$, $H>0$ such that $\|d\| \ge
2(1-1/r)r^{H}$.

We now fix $\delta,\rho>0$ such that $\rho/\delta < 4\wedge 2\|d\|$,
and choose
$$
        r = \frac{4}{4-\rho/\delta},\qquad
        H = \frac{\log(2\|d\|\delta/\rho)}{\log r}.
$$
Clearly $r>1$ and $H>0$.  Moreover, note that our choice of $r$ and $H$ 
implies that $\|d\|=2(1-1/r)r^H$ and $\rho=2r^{-H}\delta\|d\|$.  We
have therefore shown that
$$
        N_{[]}(T\cap B(t_0,\delta),
        \rho) \le 
        \bigg(
        \frac{8C\delta}{\rho}
        \bigg)^{q+1}
$$
for all $\delta,\rho>0$ such that $\rho/\delta < 4\wedge 2\|d\|$.
\end{proof}

\section{The local entropy of mixtures}
\label{sec:mixtures}

\subsection{Definitions and main results}
\label{sec:mixturesTh}

Let $\mu$ be the Lebesgue measure on $\mathbb{R}^d$. We 
fix a positive probability density $f_0$ with respect to $\mu$
($f_0>0$ and $\int f_0 d\mu=1$), 
and consider mixtures (finite convex combinations) of densities 
in the class
$$
        \{
        f_\theta:\theta\in\mathbb{R}^d
        \},\qquad
        f_\theta(x) = f_0(x-\theta)\quad\forall\,x\in\mathbb{R}^d.
$$
In everything that follows we fix a nondegenerate mixture $f^\star$ 
of the form
$$
        f^\star=\sum_{i=1}^{q^\star}\pi_i^\star f_{\theta_i^\star}.
$$
Nondegenerate means that $\pi_i^\star>0$ for all $i$, and
$\theta_i^\star\ne\theta_j^\star$ for all $i\ne j$.  

Let $\Theta\subset\mathbb{R}^d$ be a bounded parameter set such that
$\{\theta^\star_i:i=1,\ldots,q^\star\}\subseteq\Theta$, and denote 
its diameter by $2T$ (that is, $\Theta$ is included in some closed 
Euclidean ball of radius $T$).  We consider for $q\ge 1$ the 
family of $q$-mixtures
$$
        \mathcal{M}_q = \bigg\{\sum_{i=1}^q\pi_if_{\theta_i}:
        \pi_i\ge 0,~\sum_{i=1}^q\pi_i=1,~\theta_i\in\Theta\bigg\}.
$$
The goal of this section is to obtain a local entropy bound for
$\mathcal{M}_q$ at the point $f^\star$, where $\mathcal{M}_q$ is
endowed with the Hellinger metric
$$
        h(f,g) = \bigg[\int \big(\textstyle{\sqrt{f}-\sqrt{g}}
	\big)^2d\mu\bigg]^{1/2},
	\qquad f,g\in\mathcal{M}_q.
$$ 
That is, we seek bounds on quantities such as 
$N_h(\{f\in\mathcal{M}_q:h(f,f^\star)\le \varepsilon\},\delta)$, where 
$N_h$ denotes the covering number in the metric space 
$(\mathcal{M}_q,h)$ (i.e., covering by Hellinger balls). In fact, we 
prove a stronger bound of bracketing type. Our choice of the Hellinger 
metric and the particular form of the bracketing number to be considered 
is directly motivated by statistical applications \cite[ch.\ 7]{vdG00}, 
\cite[\S 7.4]{Mas07}; see \cite{GasHanOrder,GasRouHMM} for 
statistical applications of the results below.  We will adhere to this 
setting for concreteness, though other metrics may similarly be 
considered.

In the sequel, we denote by $\|\cdot\|_p$ the $L^p(f^\star 
d\mu)$-norm, that is, $\|g\|_p^p = \int |g|^p f^\star d\mu$. Note that 
the Hellinger metric can be written as $h(f,g) = 
\|\sqrt{f/f^\star}-\sqrt{g/f^\star}\|_2$.  To obtain covering bounds
for $\mathcal{M}_q$ in the Hellinger metric, we can therefore apply
the results of section \ref{sec:entropy} for the case where 
$(X,\|\cdot\|)$ is the Banach lattice $(L^2(f^\star d\mu),\|\cdot\|_2)$,
$T=\{\sqrt{f/f^\star}:f\in\mathcal{M}_q\}$, and $t_0=1$.  Indeed,
it is easily seen that\footnote{
	It is an artefact of our definitions that
	the centers of the balls that define the minimal cover of cardinality
	$N(\mathcal{H}_q(\varepsilon),\delta)$ need not lie in the
	set $\{\sqrt{f/f^\star}:f\in\mathcal{M}_q\}$, while the centers
	of the balls in the minimal cover associated to
	$N_h(\{f\in\mathcal{M}_q:h(f,f^\star)\le \varepsilon\},\delta)$
	must lie in $\mathcal{M}_q$.  This accounts for the additional
	factor $2$ in the inequality
	$N_h(\{f\in\mathcal{M}_q:h(f,f^\star)\le \varepsilon\},2\delta)\le
	N(\mathcal{H}_q(\varepsilon),\delta)$.
}
$$
	N_h(\{f\in\mathcal{M}_q:h(f,f^\star)\le \varepsilon\},2\delta) \le
	N(\mathcal{H}_q(\varepsilon),\delta) \le
	N_{[]}(\mathcal{H}_q(\varepsilon),\delta),
$$
where we have defined
$$
	\mathcal{H}_q(\varepsilon)=
	\{\sqrt{f/f^\star}:f\in\mathcal{M}_q,~ 
	\|\sqrt{f/f^\star}-1\|_2\le\varepsilon\}
	\subset L^2(f^\star d\mu).
$$
Our aim is to obtain a polynomial bound for the bracketing number
$N_{[]}(\mathcal{H}_q(\varepsilon),\delta)$.  To this end,
we will apply Theorem \ref{thm:localglobal} to the weighted class
$\mathcal{D}_q$ defined by
$$
	\mathcal{D}_q=\{d_f:f\in\mathcal{M}_q,~f\ne f^\star\},
	\qquad\quad
	d_f = \frac{\sqrt{f/f^\star}-1}{\|\sqrt{f/f^\star}-1\|_2}.
$$
The essential difficulty is now to control the global entropy of
$\mathcal{D}_q$.

The following notation will be used throughout:
\begin{align*}
        H_{0}(x)&=\sup_{\theta\in\Theta}f_{\theta}(x)/f^{\star}(x),\\
        H_{1}(x)&=\sup_{\theta\in\Theta}
        \max_{i=1,\ldots,d}| 
        \partial f_\theta(x) / \partial\theta^i|/f^{\star}(x),\\        H_{2}(x)&=\sup_{\theta\in\Theta}
        \max_{i,j=1,\ldots,d}| 
        \partial^{2} f_\theta(x) / \partial\theta^i\partial\theta^j|/f^{\star}(x),
        \\        H_{3}(x)&=
        \sup_{\theta\in\Theta}\max_{i,j,k=1,\ldots,d}| 
        \partial^{3} f_\theta(x) / \partial\theta^i\partial\theta^j\partial\theta^k|/f^{\star}(x)
\end{align*}
when $f_0$ is sufficiently differentiable,
$\mathcal{M}=\bigcup_{q\ge 1}\mathcal{M}_q$, and
$\mathcal{D}=\bigcup_{q\ge 1}\mathcal{D}_q$.

\begin{aspta}
The following hold:
\begin{enumerate}
\item $f_0\in C^3$ and $f_0(x)$, $(\partial f_0/\partial\theta^i)(x)$ vanish as 
$\|x\|\to\infty$.
\item $H_{k}\in L^{4}(f^{\star}d\mu)$ for $k=0,1,2$ and $H_{3}\in L^{2}(f^{\star}d\mu)$.
\end{enumerate} 
\end{aspta}

We can now state our main result, whose proof is given in section 
\ref{sec:proofmainmixtures}.

\begin{theorem}
\label{thm:mainmixtures}
Suppose that Assumption A holds.  Then there exist constants
$C^\star$ and $\delta^\star$, which depend on $d$, $q^\star$ and $f^\star$
but not on $\Theta$, $q$ or $\delta$, such that
$$
        N_{[]}(\mathcal{D}_{q},\delta) \le
        \bigg(
        \frac{
        C^\star (T\vee 1)^{1/3} 
        (\|H_0\|_4^4\vee\|H_1\|_4^4\vee\|H_2\|_4^4\vee\|H_3\|_2^2)
        }{\delta}
        \bigg)^{10(d+1)q} 
$$
for all $q\ge q^\star$, $\delta\le\delta^\star$.
Moreover, there is a function $D\in L^4(f^\star d\mu)$ with
$$
        \|D\|_4 \le K^\star (\|H_0\|_4\vee\|H_1\|_4\vee\|H_2\|_4),
$$
where $K^\star$ depends only on $d$ and $f^\star$, such that 
$|d|\le D$ for all $d\in\mathcal{D}$.
\end{theorem}

\begin{remark}
Assumption A is essentially a smoothness assumption on $f_0$.  Some
sort of smoothness is certainly needed for a result such as
Theorem \ref{thm:mainmixtures} to hold: see \cite[\S 3]{Ciu02} for
a counterexample in the non-smooth case.
\end{remark}

Combining Theorems 
\ref{thm:localglobal} and \ref{thm:mainmixtures}, we immediately obtain 
a local entropy bound.

\begin{theorem}
\label{thm:localmixtures}
Suppose that Assumption A holds.  Then
$$
        N_{[]}(\mathcal{H}_{q}(\varepsilon),\delta) \le
        \bigg(
        \frac{C_\Theta\,\varepsilon}{\delta}
        \bigg)^{10(d+1)q+1} 
$$
for all $q\ge q^\star$ and $\delta/\varepsilon\le 1$, where
$$
        C_\Theta = L^\star\, (T\vee 1)^{1/3} \,
        (\|H_0\|_4^4\vee\|H_1\|_4^4\vee\|H_2\|_4^4\vee\|H_3\|_2^2)^{5/4}
$$
and $L^\star$ is a constant that depends only on $d$, $q^\star$ and $f^\star$.
\end{theorem}

To illustrate these results, let us consider the important case of 
Gaussian location mixtures, which are widely used in applications
(see, e.g., \cite{GW00,GvdV01,MM11}).

\begin{example}[Gaussian mixtures]
\label{ex:gaussian}
Consider mixtures of standard Gaussian densities
$f_0(x)=(2\pi)^{-d/2}e^{-\|x\|^2/2}$, and let
$\Theta(T) = \{\theta\in\mathbb{R}^d:\|\theta\|\le T\}$.  Fix a nondegenerate
mixture $f^\star$, and define $T^\star=
\max_{i=1,\ldots,q^\star}\|\theta_i^\star\|$.  Denote by
$\mathcal{H}_q(\varepsilon,T)$ the Hellinger ball associated to the
parameter set $\Theta(T)$.  Then 
$$
        N_{[]}(\mathcal{H}_{q}(\varepsilon,T),\delta) \le
        \bigg(
        \frac{C_1^\star e^{C_2^\star T^2}\varepsilon}{\delta}
        \bigg)^{10(d+1)q+1} 
$$
for all $q\ge q^\star$, $T\ge T^\star$, and $\delta/\varepsilon\le 1$,
where $C_1^\star,C_2^\star$ are constants that depend on $d$, $q^\star$ and
$f^\star$ only.  To prove this, it evidently suffices to show that
Assumption A holds and that $\|H_k\|_4$ for $k=0,1,2$ and $\|H_3\|_2$
are of order $e^{CT^2}$.  These facts are readily verified
by a straightforward computation.
\end{example}

Let us emphasize a key feature of Theorems \ref{thm:mainmixtures} and 
\ref{thm:localmixtures}: the dependence of the entropy bounds on the 
order $q$ and on the parameter set $\Theta$ is explicit (see, e.g.,
Example \ref{ex:gaussian}).  In particular, we find that for every 
$f^\star$, the local doubling dimension of $\mathcal{M}_q$ at $f^\star$ 
is of the same order as the dimension of the natural parameter set for 
mixtures $\Delta_{q-1}\times\Theta^q$, which answers the basic question
posed in the introduction.  Obtaining this explicit dependence, which is
important in applications \cite{GasHanOrder}, is one of the main
technical challenges of the proof.  In order to show only
that $N_{[]}(\mathcal{H}_{q}(\varepsilon),\delta)$ is polynomial
in $\varepsilon/\delta$ without explicit control of the order, the proof
could be simplified and substantially generalized---see Remark 
\ref{rem:mixother} below for some discussion.
In contrast to the dependence on $q$ and $\Theta$, however, the proofs of 
Theorems \ref{thm:mainmixtures} and \ref{thm:localmixtures} do not 
provide any control of the dependence of the constants on $f^\star$. In 
particular, while we can control the local doubling dimension of 
$\mathcal{M}_q$ at $f^\star$ in terms of $q$, we do not know whether the 
dependence on $f^\star$ can be eliminated.

\begin{remark}
We have not optimized the constants in 
Theorem \ref{thm:mainmixtures} and Theorem \ref{thm:localmixtures}.
In particular, the constant $10$ in the exponent can likely be improved.
On the other hand, it is unclear whether the dependence on the diameter of 
$\Theta$ is optimal.  Indeed, if one is only interested in
global entropy $N_{[]}(\mathcal{H}_q,\delta)$ where
$\mathcal{H}_q=\{\sqrt{f/f^\star}:f\in\mathcal{M}_q\}$, then it can be
read off from the proof of Theorem \ref{thm:mainmixtures} that the constants
in the entropy bound depend on $\|H_0\|_1$ and $\|H_1\|_1$ only, which are
easily seen to scale polynomially in $T$ due to the translation invariance 
of the Lebesgue measure.  Therefore, for example in the case of Gaussian
mixtures, one can obtain a \emph{global} entropy bound which scales only
polynomially as a function of $T$, whereas the above \emph{local} entropy
bound scales as $e^{CT^2}$.  The behavior of local entropies is much more
delicate than that of global entropies, however, and we do not know whether 
it is possible to obtain a local entropy bound that scales polynomially in $T$
for the Hellinger metric.  On the other hand, if $\mathcal{M}_q$ is endowed
with the total variation metric $d_{\rm TV}(f,g)=\int |f-g|d\mu$ rather than
the Hellinger metric, then an easy modification of our proof 
yields a local entropy bound that depends only on
$\|H_i\|_1$ ($i=0,\ldots,3$), and therefore scales polynomially in $T$.
In this case the scaling matches that of the
global entropy, and is therefore optimal.
\end{remark}

\begin{remark} 
\label{rem:mixother}
The problems that we address in this section could be investigated 
in a more general setting. Let $\mathcal{F}=\{f_\theta:\theta\in\Theta\}$
be a given family of probability densities (where 
$\Theta$ is a bounded subset of $\mathbb{R}^d$), and define
$$
	\mathcal{M}_q = \bigg\{
	\sum_{i=1}^q\pi_if_{\theta_i}:\pi_i\ge 0,~
	\sum_{i=1}^q\pi_i=1,~\theta_i\in\Theta
	\bigg\}.
$$
The case that we have considered corresponds to the choice
$\mathcal{F}=\{f_0(\,\cdot\,-\theta):\theta\in\Theta\}$,
but in principle any parametrized family $\mathcal{F}$ may be considered.

Remarkably, most of the proof of Theorem \ref{thm:mainmixtures} does not 
rely at all on the specific choice of $\mathcal{F}$, so that very 
similar techniques may be used to study more general mixtures. The only 
point where the structure of $\mathcal{F}$ has been used is in the local 
geometry Theorem \ref{thm:localgeom} below, whose proof (using Fourier 
methods) relies on the specific form of location mixtures.  We believe 
that essentially the same result holds more generally, but a different
method of proof would likely be needed.

The proof of Theorem \ref{thm:localgeom} below is rather technical: the 
difficulty lies in the fact that the result holds uniformly in the order 
$q$. This is necessary in order to obtain bounds in Theorems 
\ref{thm:mainmixtures} and \ref{thm:localmixtures} that depend 
explicitly on $q$.  If the explicit dependence on $q$ is not needed, 
then our proof of Theorem \ref{thm:localgeom} can be simplified and 
adapted to hold for much more general classes $\mathcal{F}$, see 
\cite{GasRouHMM}. 

Finally, we note that $\mathcal{M}=\bigcup_q\mathcal{M}_q$ is simply the 
convex hull of $\mathcal{F}$.  The problem of estimating the metric 
entropy of convex hulls has been widely studied 
\cite{CKP99,Gao01,Gao04,GW00,GvdV01}.  In general, however, the convex 
hull is infinite-dimensional, so that this problem is quite distinct 
from the problems we have considered.
\end{remark}

\begin{remark}
Weighted entropy bounds as in Theorem \ref{thm:mainmixtures} are of 
independent interest.  A qualitative version of this bound (without 
uniform control in $q$ and $T$) was assumed in \cite{DCG99}, 
which provided inspiration for the present effort.  However, in 
\cite[Prop.\ 3.1]{DCG99}, it is assumed without justification that one 
can choose a multiplicative rather than additive remainder term in a 
Taylor expansion.  The requisite justification is provided (in a much 
more precise form) by the local geometry theorem to be described 
presently.  Developing a precise understanding of the local geometry
of mixtures is the fundamental challenge to be surmounted in our
setting, and our
local geometry result therefore constitutes the central contribution 
of this paper. 
\end{remark}

\subsection{The local geometry of mixtures}
\label{sec:mixturesGeo}

At the heart of the proof of Theorem \ref{thm:mainmixtures} lies a 
result on the local geometry of location mixtures, Theorem 
\ref{thm:localgeom} below.  Before we can develop this result, we
must introduce some notation.

Define the Euclidean balls $B(\theta,\varepsilon)=\{\theta'\in\mathbb{R}^d: 
\|\theta-\theta'\|<\varepsilon\}$, denote by $\langle u,v\rangle$ the 
inner product of two vectors $u,v\in\mathbb{R}^d$, and denote by 
$\langle A,u\rangle= \{\langle\theta,u\rangle:\theta\in A\}\subseteq\mathbb{R}$
the inner product of a set $A\subseteq\mathbb{R}^d$ with a vector 
$u\in\mathbb{R}^d$.

\begin{lemma}
\label{lem:bubble}
It is possible to choose a bounded convex neighborhood $A_i$ of 
$\theta_i^\star$ for every $i=1,\ldots,q^\star$ such that, for some
linearly independent family $u_1,\ldots,u_d\in\mathbb{R}^d$,
the sets $\{\langle A_i,u_j\rangle:i=1,\ldots,q^\star\}$ are disjoint
for every $j=1,\ldots,d$.
\end{lemma}

\begin{proof}
We first claim that one can choose linearly independent
$u_1,\ldots,u_d$ such that $|\{\langle\theta_i^\star,u_j\rangle:
i=1,\ldots,q^\star\}|=q^\star$ for every $j=1,\ldots,d$.  Indeed,
note that the set
$\{u\in\mathbb{R}^d:|\{\langle\theta_i^\star,u\rangle:
i=1,\ldots,q^\star\}|<q^\star\}$ is a finite union of $(d-1)$-dimensional
hyperplanes, which has Lebesgue measure zero.  Therefore, if we draw
a rotation matrix $T$ at random from the Haar measure on
$\mathrm{SO}(d)$, and let $u_i=Te_i$ for all $i=1,\ldots,d$ 
where $\{e_1,\ldots,e_d\}$ is the standard Euclidean basis in
$\mathbb{R}^d$, then the desired property will hold with unit
probability.  To complete the proof, it suffices to choose
$A_i = B(\theta_i^\star,\varepsilon/4)$ with
$\varepsilon = \min_k\min_{i\ne j}|\langle\theta_i^\star-
\theta_j^\star,u_k\rangle|$.
\end{proof}

We now fix once and for all a family of neighborhoods 
$A_1,\ldots,A_{q^\star}$ as in Lemma \ref{lem:bubble}.  The precise 
choice of these sets only affects the constants in the proofs below and 
is therefore irrelevant to our final result; we only presume that 
$A_1,\ldots,A_{q^\star}$ remain fixed throughout the proofs.  Let us 
also define $A_0=\mathbb{R}^d\backslash(A_1\cup\cdots\cup A_{q^\star})$. 
Then $\{A_0,\ldots,A_{q^\star}\}$ partitions the parameter set 
$\mathbb{R}^d$ in such a way that each bounded element $A_i$, 
$i=1,\ldots,q^\star$ contains precisely one component of the mixture 
$f^\star$, while the unbounded element $A_0$ contains no components of 
$f^\star$.

%

Let us define for each finite measure $\lambda$ on
$\mathbb{R}^d$ the function
$$
        f_{\lambda}(x) = \int f_\theta(x)\,\lambda(d\theta).
$$
We also define the derivatives $D_1f_\theta(x)\in\mathbb{R}^d$ and
$D_2f_\theta(x)\in\mathbb{R}^{d\times d}$ as
$$
	[D_1f_\theta(x)]_{i} = \frac{\partial}{\partial\theta^i}f_\theta(x),
	\qquad
	[D_2f_\theta(x)]_{ij} = \frac{\partial^2}{\partial\theta^i
	\partial\theta^j}f_\theta(x).
$$
Denote by $\mathfrak{P}(A)$ the space of probability measures supported
on $A\subseteq\mathbb{R}^d$, and denote by $M_+^d$ the family of all 
$d\times d$ positive semidefinite (symmetric) matrices.

\begin{definition}
Let us write
\begin{multline*}
        \mathfrak{D} = \{(\eta,\beta,\rho,\tau,\nu):
	\eta_1,\ldots,\eta_{q^\star}\in\mathbb{R},~
	\beta_1,\ldots,\beta_{q^\star}\in\mathbb{R}^d,~
	\rho_1,\ldots,\rho_{q^\star}\in M_+^d,\\
	\tau_0,\ldots,\tau_{q^\star}\ge 0,~
        \nu_0\in\mathfrak{P}(A_0),\ldots,
        \nu_{q^\star}\in\mathfrak{P}(A_{q^\star})\}.
\end{multline*}
Then we define for each $(\eta,\beta,\rho,\tau,\nu)\in\mathfrak{D}$
the function
$$
        \ell(\eta,\beta,\rho,\tau,\nu) =
        \tau_0\frac{f_{\nu_0}}{f^\star}+
        \sum_{i=1}^{q^\star}\bigg\{
        \eta_i\frac{f_{\theta_i^\star}}{f^\star}
        +\beta_i^*\frac{D_1f_{\theta_i^\star}}{f^\star}
        +\mathrm{Tr}\bigg[\rho_i\frac{D_2f_{\theta_i^\star}}{f^\star}
		\bigg]
        +\tau_i\frac{f_{\nu_i}}{f^\star}
        \bigg\},
$$
and the nonnegative quantity
\begin{multline*}
        N(\eta,\beta,\rho,\tau,\nu) = \tau_0+ 
	\sum_{i=1}^{q^\star}|\eta_i+\tau_i| +
        \sum_{i=1}^{q^\star}
        \bigg\|\beta_i+\tau_i\int(\theta-\theta_i^\star)
                \,\nu_i(d\theta)\bigg\| +
	\mbox{} \\
        \sum_{i=1}^{q^\star} \mathrm{Tr}[\rho_i]
        + \sum_{i=1}^{q^\star}
	\frac{\tau_i}{2}\int\|\theta-\theta_i^\star\|^2\nu_i(d\theta).
\end{multline*}
\end{definition}

We now formulate the key result on the local geometry of the mixture 
class $\mathcal{M}$.

\begin{theorem}
\label{thm:localgeom}
Suppose that 
\begin{enumerate}
\item $f_0\in C^2$ and $f_0(x)$, $D_1f_0(x)$ vanish as $\|x\|\to\infty$.
\item $\|[D_1f_0]_i/f^\star\|_1<\infty$ and $\|[D_2f_0]_{ij}/f^\star\|_1<\infty$
for all $i,j=1,\ldots,d$.
\end{enumerate}
Then there exists a constant $c^\star>0$ such that
$$
        \|\ell(\eta,\beta,\rho,\tau,\nu)\|_1 \ge
        c^\star\,N(\eta,\beta,\rho,\tau,\nu)\quad
        \mbox{for all }(\eta,\beta,\rho,\tau,\nu)\in\mathfrak{D}.
$$
\emph{[}The constant $c^\star$ may depend on $f^\star$ and 
$A_1,\ldots,A_{q^\star}$ but not on $\eta,\beta,\rho,\tau,\nu$.\emph{]}
\end{theorem}

Before we turn to the proof, let us introduce a notion that is familiar 
in quantum mechanics.  If $(\Omega,\Sigma)$ is a measurable space, 
call the map $\lambda:\Sigma\to\mathbb{R}^{d\times d}$ a \emph{state}\footnote{
	Our terminology is in analogy with the notion of a
	state on the $C^*$-algebra $\mathbb{C}^{d\times d}\otimes
        C_{\mathbb{C}}(\Omega)$, where $\Omega$ is a compact metric space and
	$C_{\mathbb{C}}(\Omega)$ is the algebra of complex-valued continuous 
	functions on $\Omega$.  Such states can be
	represented by the complex-valued counterpart of our definition.
} if 
\begin{enumerate}
\item $A\mapsto[\lambda(A)]_{ij}$ is a signed measure for
every $i,j=1,\ldots,d$;
\item $\lambda(A)$ is a nonnegative symmetric matrix for every $A\in\Sigma$;
\item $\mathrm{Tr}[\lambda(\Omega)]=1$.
\end{enumerate}
It is easily seen that for any unit vector $\xi\in\mathbb{R}^d$, the
map $A\mapsto\langle\xi,\lambda(A)\xi\rangle$ is a sub-probability measure.
Moreover, if $\xi_1,\ldots,\xi_d\in\mathbb{R}^d$ are linearly independent,
there must be at least one $\xi_i$ such that
$\langle\xi_i,\lambda(\Omega)\xi_i\rangle>0$.  Finally, let 
$B\subset\mathbb{R}^d$ be a compact set and let $(\lambda_n)_{n\ge 0}$
be a sequence of states on $B$.  Then there exists a subsequence along
which $\lambda_n$ converges weakly to some state $\lambda$ on $B$
in the sense that $\int \mathrm{Tr}[M(\theta)\lambda_n(d\theta)]
\to \int \mathrm{Tr}[M(\theta)\lambda(d\theta)]$
for every continuous function $M:B\to\mathbb{R}^{d\times d}$.
To see this, it suffices to note that we may extract a subsequence
such that all matrix elements $[\lambda_n]_{ij}$ converge weakly
to a signed measure by the compactness of $B$, and it is evident
that the limit must again define a state.

\begin{proof}[Proof of Theorem \ref{thm:localgeom}]
Suppose that the conclusion of the theorem does not hold.  Then there must 
exist a sequence of coefficients 
$(\eta^{n},\beta^{n},\rho^{n},\tau^{n},\nu^{n})\in\mathfrak{D}$ with
$$
	\frac{\|\ell(\eta^{n},\beta^{n},\rho^{n},\tau^{n},\nu^{n})\|_1}{
	N(\eta^{n},\beta^{n},\rho^{n},\tau^{n},\nu^{n})}
	\xrightarrow{n\to\infty}0.
$$
Let us fix such a sequence throughout the proof.

Applying Taylor's theorem to 
$u\mapsto f_{\theta_i^\star+u(\theta-\theta_i^\star)}$, we can write
for $i=1,\ldots,q^\star$
\begin{align*}
        & \eta_{i}^{n}\frac{f_{\theta_{i}^\star}}{f^\star}
        + \beta_{i}^{n*} \frac{D_{1}f_{\theta_{i}^\star}}{f^\star} 
        + \mathrm{Tr}\bigg[\rho_{i}^{n}
	\frac{D_{2}f_{\theta_{i}^\star}}{f^\star}
	\bigg]
        + \tau_{i}^{n} \frac{f_{\nu_{i}^{n}}}{f^\star} 
        \\ &\mbox{}=
        \left(\eta_{i}^{n}+\tau_{i}^{n}\right)
        \frac{f_{\theta_{i}^\star}}{f^\star}
        + \bigg(\beta_{i}^{n}+\tau_{i}^{n} \int (\theta-\theta_{i}^\star)
        \,\nu_{i}^{n}(d\theta)\bigg)^*
        \frac{D_{1}f_{\theta_{i}^\star}}{f^\star} 
        + \mathrm{Tr}\bigg[\rho_{i}^{n}
	\frac{D_{2}f_{\theta_{i}^\star}}{f^\star}
	\bigg]
        \\ &\quad\mbox{}
        + \frac{\tau_{i}^{n}}{2}\int \|\theta-\theta_{i}^\star\|^{2}\,
        \nu_{i}^{n}(d\theta) \int \mathrm{Tr}\bigg[
	\bigg\{\int_{0}^{1} 
        \frac{D_{2}f_{\theta_i^\star+u(\theta-\theta_{i}^\star)}}{f^\star}\,
        2(1-u)\,du\bigg\} \,\lambda_{i}^{n}(d\theta)\bigg]
\end{align*}
where $\lambda_{i}^{n}$ is the state on $A_i$ defined by
$$
        \int \mathrm{Tr}[M(\theta)\,\lambda_{i}^{n}(d\theta)] = 
        \frac{\int \mathrm{Tr}[M(\theta)\,(\theta-\theta_{i}^\star)
	(\theta-\theta_{i}^\star)^*]\,
        \nu_{i}^{n}(d\theta)}{
        \int\|\theta-\theta_{i}^\star\|^{2}\,\nu_{i}^{n}(d\theta)}
$$
(it is clearly no loss of generality to assume that $\nu_i^n$ has
no mass at $\theta_i^\star$ for any $i,n$, so that everything is well 
defined).  We now define the coefficients
$$
        \displaylines{
        a_{i}^{n} = \frac{\eta_{i}^{n} + \tau_{i}^{n}}{
        N(\eta^{n},\beta^{n},\rho^{n},\tau^{n},\nu^{n})},\qquad
        b_{i}^{n} = \frac{\beta_{i}^{n} +
        \tau_{i}^{n}\int(\theta-\theta_{i}^\star)\,\nu_{i}^{n}(d\theta)}{
        N(\eta^{n},\beta^{n},\rho^{n},\tau^{n},\nu^{n})},
        \cr
        c_{i}^{n} = \frac{\rho_{i}^{n}}{
        N(\eta^{n},\beta^{n},\rho^{n},\tau^{n},\nu^{n})},\qquad
        d_{i}^{n} = \frac{\frac{\tau_{i}^{n}}{2}
        \int\|\theta-\theta_{i}^\star\|^{2}\,\nu_{i}^{n}(d\theta)}{
        N(\eta^{n},\beta^{n},\rho^{n},\tau^{n},\nu^{n})}
        }
$$
for $i=1,\ldots,q^\star$, and
$$
	a_0^n = \frac{\tau_{0}^{n}}{
        N(\eta^{n},\beta^{n},\rho^{n},\tau^{n},\nu^{n})}.
$$
Note that
$$
        |a_0^n| + \sum_{i=1}^{q^\star}\left\{|a_{i}^{n}|+\|b_{i}^{n}\|
        +\mathrm{Tr}[c_{i}^{n}]+|d_{i}^{n}|\right\}=1
$$
for all $n$.  We may therefore extract a subsequence such that:
\begin{enumerate}
\item There exist $a_i\in\mathbb{R}$, $b_i\in\mathbb{R}^d$, 
$c_i\in M_+^d$, and $a_0,d_i\ge 0$ (for $i=1,\ldots,q^\star$) with
$|a_0|+\sum_{i=1}^{q^\star}\left\{|a_{i}|+\|b_{i}\|+\mathrm{Tr}[c_{i}]
+|d_{i}|\right\}=1$, such that $a_0^n\to a_0$ and $a_i^n\to a_i$, 
$b_i^n\to b_i$, $c_i^n\to c_i$, $d_i^n\to d_i$ as $n\to\infty$ 
for all $i=1,\ldots,q^\star$.
\item There exists a sub-probability measure $\nu_0$ supported on
$A_0$, such that $\nu_0^n$ converges vaguely to $\nu_0$ as $n\to\infty$.
\item There exist states $\lambda_i$ supported on 
$\mathop{\mathrm{cl}}A_i$ for $i=1,\ldots,q^\star$, such that $\lambda_i^n$ 
converges weakly to $\lambda_i$ as $n\to\infty$ for every $i=1,\ldots,q^\star$.
\end{enumerate}
The functions 
$\ell(\eta^{n},\beta^{n},\rho^{n},\tau^{n},\nu^{n}) / 
N(\eta^{n},\beta^{n},\rho^{n},\tau^{n},\nu^{n})$ converge
pointwise along this subsequence to the function $h/f^\star$ defined by
\begin{multline*}
        h = a_0\,f_{\nu_0} +
        \sum_{i=1}^{q^\star}\bigg\{
        a_i\,f_{\theta_{i}^\star}
        + b_i^*\,D_{1}f_{\theta_{i}^\star} 
        + \mathrm{Tr}[c_i\,D_{2}f_{\theta_{i}^\star}] \\ \mbox{}
        + d_i\int \mathrm{Tr}\bigg[\bigg\{\int_{0}^{1} 
        D_{2}f_{\theta_i^\star+u(\theta-\theta_{i}^\star)}\,
        2(1-u)\,du\bigg\} \,\lambda_i(d\theta)\bigg]
      \,  \bigg\}.
\end{multline*}
But as $\|\ell(\eta^{n},\beta^{n},\rho^{n},\tau^{n},\nu^{n})\|_1 /
N(\eta^{n},\beta^{n},\rho^{n},\tau^{n},\nu^{n})\to 0$, we have
$\|h/f^\star\|_1=0$ by Fatou's lemma.  As $f^\star$ is strictly positive, 
we must have $h\equiv 0$.

To proceed, we need the following lemma.

\begin{lemma}
\label{lem:flem}
The Fourier transform $F[h](s):=\int e^{\ii\langle x,s\rangle}h(x)dx$
is given by
\begin{multline*}
        F[h](s) = F[f_0](s)\, \bigg[
	a_0\int e^{\ii\langle\theta,s\rangle}\,\nu_0(d\theta) + 
        \sum_{i=1}^{q^\star}\bigg\{
        a_i\,e^{\ii\langle\theta_{i}^\star,s\rangle}
        + \ii\langle b_i,s\rangle\,e^{\ii\langle\theta_{i}^\star,s\rangle} 
	\mbox{} \\
        - \langle s,c_i s\rangle\,e^{\ii\langle\theta_{i}^\star,s\rangle}
        - d_i\,e^{\ii\langle\theta_{i}^\star,s\rangle}
        \int \phi(\ii\langle\theta-\theta_{i}^\star,s\rangle)\,
	\langle s,\lambda_i(d\theta)s\rangle
        \bigg\}\bigg]
\end{multline*}
for all $s\in\mathbb{R}^d$.
Here we defined the function $\phi(u) = 2(e^u-u-1)/u^2$.
\end{lemma}

\begin{proof}
The $a_i,b_i,c_i$ terms are easily computed using integration by parts. 
It remains to compute the Fourier transform of the function
$$
        [\Xi_i(x)]_{jk} = 
        \int \bigg\{\int_{0}^{1}
        [D_{2}f_{\theta_i^\star+u(\theta-\theta_{i}^\star)}(x)]_{jk}\,
        2(1-u)\,du\bigg\} \,[\lambda_i(d\theta)]_{kj}.
$$
We begin by noting that
\begin{multline*}
        \int \int \int_{0}^{1}
        |[D_{2}f_{\theta_i^\star+u(\theta-\theta_{i}^\star)}(x)]_{jk}|\,
        2(1-u)\,du\,dx\,|[\lambda_i]_{kj}|(d\theta) = \\
        \|[\lambda_i]_{kj}\|_{\rm TV}\int |[D_2f_0(x)]_{jk}|\,dx < \infty.
\end{multline*}
We may therefore apply Fubini's theorem, giving
\begin{align*}
        F[[\Xi_i]_{jk}](s) &= 
	- F[f_0](s)\,
	s_j s_k\, e^{\ii\langle\theta_i^\star,s\rangle}
        \int\bigg\{\int_{0}^{1}
        e^{\ii u\langle\theta-\theta_{i}^\star,s\rangle}
        2(1-u) du\bigg\} [\lambda_i(d\theta)]_{kj} \\ \mbox{}
        &= 
	- F[f_0](s)\,
	s_j s_k\,e^{\ii\langle\theta_i^\star,s\rangle}
        \int \phi(\ii\langle\theta-\theta_i^\star,s\rangle)\,
	[\lambda_i(d\theta)]_{kj},
\end{align*}
where we have computed the inner integral using integration by parts.
\end{proof}

Let $u_1,\ldots,u_d\in\mathbb{R}^d$ be a linearly independent family
satisfying the condition of Lemma \ref{lem:bubble}.
As $F[h](s)=0$ for all $s\in\mathbb{R}^d$, we obtain
$$
        \Phi^\ell(\ii t) :=
	a_0\,\Phi_0^\ell(\ii t)+
        \sum_{i=1}^{q^\star}
	e^{\ii t\langle\theta_{i}^\star,u_\ell\rangle}
	\big\{
        a_i + \ii t \langle b_i,u_\ell\rangle 
	- t^2 \langle u_\ell,c_i u_\ell\rangle 
	- d_i\, t^2\, \Phi_i^\ell(\ii t)
        \big\} = 0
$$
for all $\ell=1,\ldots,d$ and $t\in[-\iota,\iota]\subset\mathbb{R}$
for some $\iota>0$, where we defined
$$
        \Phi_i^\ell(\ii t) = 
        \int \phi(\ii t\langle\theta-\theta_{i}^\star,u_\ell\rangle)\,
	\langle u_\ell,\lambda_i(d\theta)u_\ell\rangle
$$
for $i=1,\ldots,q^\star$, and
$$
	\Phi_0^\ell(\ii t) = 
	\int e^{\ii t\langle\theta,u_\ell\rangle}\,\nu_0(d\theta).
$$
Indeed, it suffices to note that $F[f_0](0)=1$ and that $s\mapsto 
F[f_0](s)$ is continuous, so that this claim follows from Lemma 
\ref{lem:flem} and the fact that $F[f_0](s)$ is nonvanishing in a 
sufficiently small neighborhood of the origin.

As all $\lambda_i$ have compact support, it is easily seen that for every 
$i=1,\ldots,q^\star$, the function $\Phi_i^\ell(z)$ is defined for all 
$z\in\mathbb{C}$ by a convergent power series.  The function 
$\Psi^\ell(\ii t):=\Phi^\ell(\ii t)-a_0\,\Phi_0^\ell(\ii t)$ is 
therefore an entire function with $|\Psi^\ell(z)|\le k_1e^{k_2|z|}$ for 
some $k_1,k_2>0$ and all $z\in\mathbb{C}$.  But as $\Phi^\ell(\ii t)=0$ 
for $t\in[-\iota,\iota]$, it follows from \cite{Luk70}, Theorem 7.2.2 
that $a_0\,\Phi_0^\ell(\ii t)$ is the Fourier transform of a finite 
measure with compact support. Thus we may assume without loss of 
generality that the law of $\langle\theta,u_\ell\rangle$ under the 
sub-probability $\nu_0$ is compactly supported for every 
$\ell=1,\ldots,d$, so by linear independence $\nu_0$ must be compactly 
supported.  Therefore, the function $\Phi^\ell(z)$ is defined for all 
$z\in\mathbb{C}$ by a convergent power series.  But as $\Phi^\ell(z)$ 
vanishes for $z\in\ii[-\iota,\iota]$, we must have $\Phi^\ell(z)=0$ for 
all $z\in\mathbb{C}$, and in particular
\begin{equation}
\label{eq:laplace}
        \Phi^\ell(t) =
	a_0\,\Phi_0^\ell(t)+
        \sum_{i=1}^{q^\star}
	e^{t\langle\theta_{i}^\star,u_\ell\rangle}
	\big\{
        a_i + t \langle b_i,u_\ell\rangle 
	+ t^2 \langle u_\ell,c_i u_\ell\rangle + d_i\, t^2\, \Phi_i^\ell(t)
        \big\} = 0
\end{equation}
for all $t\in\mathbb{R}$ and $\ell=1,\ldots,d$.
In the remainder of the proof, we argue that (\ref{eq:laplace}) can not 
hold, thus completing the proof by contradiction.

At the heart of our proof is an inductive argument.  Recall that
by construction, the projections $\{\langle A_i,u_\ell\rangle:
i=1,\ldots,q^\star\}$ are disjoint open intervals in $\mathbb{R}$
for every $\ell=1,\ldots,d$.
We can therefore relabel them in increasing order: that is, define
$(\ell1),\ldots,(\ell q^\star)\in\{1,\ldots,q^\star\}$ so that
$\langle \theta_{(\ell1)}^\star,u_\ell\rangle <
\langle \theta_{(\ell2)}^\star,u_\ell\rangle < \cdots <
\langle \theta_{(\ell q^\star)}^\star,u_\ell\rangle$.  The following
key result provides the inductive step in our proof.

\begin{prop} 
\label{prop:finduction}
Fix $\ell\in\{1,\ldots,d\}$, and define 
$$
	\tilde\Phi_0^\ell(t) :=
	a_0\,\Phi_0^\ell(t)+\sum_{i=1}^{q^\star}
	a_i\,e^{t\langle\theta_{i}^\star,u_\ell\rangle}.
$$
Suppose that for some $j\in\{1,\ldots,q^\star\}$ we have
$\Phi^{\ell,j}(t)=0$ for all $t\in\mathbb{R}$, where
$$
	\Phi^{\ell,j}(t):=
	\tilde\Phi_0^\ell(t)+
        \sum_{i=1}^{j}
	e^{t\langle\theta_{(\ell i)}^\star,u_\ell\rangle}
	\big\{
        t \langle b_{(\ell i)},u_\ell\rangle 
	+ t^2 \langle u_\ell,c_{(\ell i)} u_\ell\rangle 
	+ d_{(\ell i)}\, t^2\, \Phi_{(\ell i)}^\ell(t)
        \big\}.
$$
Then
$d_{(\ell j)}\langle u_\ell,\lambda_{(\ell j)}(\mathbb{R}^d)u_\ell\rangle=0$, 
$\langle u_\ell,c_{(\ell j)} u_\ell\rangle=0$, and 
$\langle b_{(\ell j)},u_\ell\rangle=0$.
\end{prop}

\begin{proof} 
Let us write for simplicity $\theta_i^\ell = 
\langle\theta_i^\star,u_\ell\rangle$, and denote by $\lambda_i^\ell$ 
and $\nu_0^\ell$ the finite measures on $\mathbb{R}$ defined such that 
$\int f(x)\lambda_i^\ell(dx) = \int f(\langle\theta,u_\ell\rangle) \langle 
u_\ell,\lambda_i(d\theta)u_\ell\rangle$ and $\int f(x)\nu_0^\ell(dx) = 
\int f(\langle\theta,u_\ell\rangle) \nu_0(d\theta)$, respectively. For 
notational convenience, we will assume in the following that $(\ell 
i)=i$ and $\nu_0^\ell(\{\theta_i^\ell\})=0$ for all 
$i=1,\ldots,q^\star$.  This entails no loss of generality: the former 
can always be attained by relabeling of the points $\theta_i^\star$, 
while $\tilde\Phi_0^\ell$ is unchanged if we replace $\nu_0^\ell$ and 
$a_i$ by 
$\nu_0^\ell(\,\cdot\,\cap\mathbb{R}\backslash\{\theta_1^\ell,\ldots, 
\theta_{q^\star}^\ell\})$ and $a_i + 
a_0\,\nu_0^\ell(\{\theta_i^\ell\})$, respectively.  Note that
$$
	\langle A_i,u_\ell\rangle = 
	\mbox{}]\theta_i^{\ell-}, \theta_i^{\ell+}[\mbox{},
	\quad\mbox{where}\quad 
	\theta_i^{\ell-}<\theta_i^\ell< 
	\theta_i^{\ell+}<\theta_{i+1}^{\ell-}\quad\mbox{for all }i
$$
by our assumptions ($\langle A_i,u_\ell\rangle$ must be an interval
as $A_i$ is convex).

\textbf{Step 1}.  We claim that the following hold:
$$
	a_i = 0\mbox{ for all }i\ge j+1\quad\mbox{and}\quad
	a_0\,\nu_0^\ell([\theta_{j+1}^\ell,\infty[\mbox{})=0.
$$
Indeed, suppose this is not the case.  Then it is easily seen that
$$
	\liminf_{t\to\infty} \frac{|\tilde\Phi_0^\ell(t)|}{
	e^{t\theta_{j+1}^\ell}}>0,
$$
where we have used that $\nu_0^\ell$ has no mass at
$\{\theta_1^\ell,\ldots,\theta_{q^\star}^\ell\}$.  On the other hand,
as $\phi$ is positive and increasing and as $\lambda_i$ is supported
on $\mathop{\mathrm{cl}}A_i$, we can estimate
$$
	0\le
	\frac{t^2\,e^{t\theta_i^\ell}\,\Phi_i^\ell(t)}{e^{t\theta_{j+1}^\ell}}
	\le t^2\,e^{-t(\theta_{j+1}^\ell-\theta_i^\ell)}\,
	\phi(t\{\theta_j^{\ell+}-\theta_i^\ell\})\,
	\lambda_i^\ell(\mathbb{R})
	\xrightarrow{t\to\infty}0
$$
for $i=1,\ldots,j$.  But then we must have
$$
	0 = 
	\liminf_{t\to\infty} \frac{|\Phi^{\ell,j}(t)|}{
	e^{t\theta_{j+1}^\ell}}>0,
$$
which yields the desired contradiction.

\textbf{Step 2}.  We claim that the following hold:
$$
	d_j\lambda_j^\ell([\theta_j^\ell,\infty[\mbox{})=0,
	\quad \langle u_\ell,c_j u_\ell\rangle = 0,\quad\mbox{and}\quad
	a_0\,\nu_0^\ell([\theta_j^\ell,\infty[\mbox{})=0.
$$
Indeed, suppose this is not the case.  As $\nu_0^\ell(\{\theta_j^\ell\})=0$,
we can choose $\varepsilon>0$ such that
$\nu_0^\ell([\theta_j^\ell+\varepsilon,\infty[\mbox{})\ge
\nu_0^\ell([\theta_j^\ell,\infty[\mbox{})/2$.  As $a_0,d_j\ge 0$, and using
that $\phi$ is positive and increasing with $\phi(0)=1$ and that
$e^{\varepsilon t}\ge (\varepsilon t)^2/2$ for $t\ge 0$, 
we can estimate
\begin{multline*}
	a_0\,\Phi_0^\ell(t)+ e^{t\theta_j^\ell}\big\{
	t^2\langle u_\ell,c_j u_\ell\rangle+
	d_j\,t^2\,\Phi_j^\ell(t)\big\} \ge \mbox{}\\
	t^2\,e^{t\theta_j^\ell}\,\bigg\{
	\frac{\varepsilon^2}{4}\,
	a_0\,\nu_0^\ell([\theta_j^\ell,\infty[\mbox{})
	+ \langle u_\ell,c_j u_\ell\rangle +
	d_j\,\lambda_j^\ell([\theta_j^\ell,\infty[\mbox{})\bigg\} > 0
\end{multline*}
for all $t\ge 0$.  On the other hand, it is easily seen that
$$
	\frac{1}{t^2\,e^{t\theta_j^\ell}}
	\left[
        \sum_{i=1}^{j}
	e^{t\theta_i^\ell}
	\big\{
	a_i
        + t \langle b_i,u_\ell\rangle \big\}
	+
        \sum_{i=1}^{j-1}
	e^{t\theta_i^\ell}
	\big\{
	t^2 \langle u_\ell,c_{i} u_\ell\rangle
	+ d_i\, t^2\, \Phi_i^\ell(t)
        \big\}\right] \xrightarrow{t\to\infty} 0.
$$
But this would imply that
$$
	0 = \lim_{t\to\infty}
	\frac{\Phi^{\ell,j}(t)}{a_0\,\Phi_0^\ell(t)+ e^{t\theta_j^\ell}\{
        t^2\langle u_\ell,c_j u_\ell\rangle+
        d_j\,t^2\,\Phi_j^\ell(t)\}} = 1,
$$
which yields the desired contradiction.

\textbf{Step 3}.  We claim that the following hold:
$$
	d_j\,\lambda_j^\ell([\theta_j^{\ell-},\theta_j^\ell[\mbox{})=0
	\quad\mbox{and}\quad
	a_0\,\nu_0^\ell([\theta_j^{\ell-},\theta_j^\ell[\mbox{})=0.
$$
Indeed, suppose this is not the case.  We can compute
\begin{multline*}
        0 = \frac{d^2}{dt^2}\left(
        \frac{\Phi^{\ell,j}(t)}{e^{t\theta_j^\ell}}\right)
        = d_j\int e^{t(\theta-\theta_j^\ell)}\,
        \lambda_j^\ell(d\theta) 
	+ a_0 \int e^{t(\theta-\theta_j^\ell)}\,
	(\theta-\theta_j^\ell)^2\,\nu_0^\ell(d\theta) \\
	\mbox{} +
        \sum_{i=1}^{j-1}
        \frac{d^2}{dt^2}\,
	e^{-t(\theta_j^\ell-\theta_i^\ell)}\big\{
        a_i + t\langle b_i,u_\ell\rangle 
	+ t^2\langle u_\ell,c_i u_\ell\rangle 
	+ d_i\,t^2\,\Phi_i^\ell(t)\big\},
\end{multline*}
where the derivative and integral may be exchanged by \cite{Wil91}, 
Appendix A16.  We now note that as $a_0,d_j\ge 0$, we can estimate
for $t\ge 0$
\begin{multline*}
        d_j\int e^{t(\theta-\theta_j^\ell)}\,
        \lambda_j^\ell(d\theta) 
	+ a_0 \int e^{t(\theta-\theta_j^\ell)}\,
	(\theta-\theta_j^\ell)^2\,\nu_0^\ell(d\theta) \ge \mbox{}\\
	e^{t(\theta_j^{\ell-}-\theta_j^\ell)}
	\Bigg\{
	d_j\,\lambda_j^\ell([\theta_j^{\ell-},\theta_j^\ell[\mbox{})
	+
	a_0\int_{[\theta_j^{\ell-},\theta_j^\ell[}
	(\theta-\theta_j^\ell)^2\,\nu_0^\ell(d\theta)
	\Bigg\}>0.
\end{multline*}
On the other hand, as $(e^x-1)/x$ is positive and increasing, we obtain
for $t\ge 0$
\begin{align*}
&       e^{-t(\theta_j^{\ell-}-\theta_j^\ell)}
        \left|\frac{d^2}{dt^2}\,
        e^{-t(\theta_j^\ell-\theta_i^\ell)}\,t^2\,\Phi_i^\ell(t)\right| 
\\ &\quad \mbox{} 
        = e^{-t(\theta_j^{\ell-}-\theta_j^\ell)}
        \times e^{-t(\theta_j^\ell-\theta_i^\ell)} \times
        \Bigg|
        (\theta_j^\ell-\theta_i^\ell)^2
        \int t^2\phi(t\{\theta-\theta_i^\ell\})\,\lambda_i^\ell(d\theta)
\\ &\qquad\qquad \mbox{} 
        -2(\theta_j^\ell-\theta_i^\ell)
        \int 
        \frac{e^{t(\theta-\theta_i^\ell)}-1}{\theta-\theta_i^\ell}
        \,\lambda_i^\ell(d\theta)
        + \int e^{t(\theta-\theta_i^\ell)}\,\lambda_i^\ell(d\theta)
        \Bigg| 
\\ &\quad \mbox{} 
        \le e^{-t(\theta_j^{\ell-}-\theta_i^\ell)}
        \Bigg\{
        (\theta_j^\ell-\theta_i^\ell)^2\,t^2\,
        \phi(t\{\theta_i^{\ell+}-\theta_i^\ell\}) 
\\ &\qquad\qquad\qquad\qquad\qquad \mbox{}
        + 2\,(\theta_j^\ell-\theta_i^\ell)\,
        \frac{e^{t(\theta_i^{\ell+}-\theta_i^\ell)}-1}{
                \theta_i^{\ell+}-\theta_i^\ell}
        + e^{t(\theta_i^{\ell+}-\theta_i^\ell)}
        \Bigg\}
	\,\lambda_i^\ell(\mathbb{R}),
\end{align*}
which converges to zero as $t\to\infty$ for every $i<j$.  It follows that
$$
        0 = \lim_{t\to\infty}\frac{\frac{d^2}{dt^2}\left(
        \Phi^{\ell,j}(t)/e^{t\theta_j^\ell}\right)}{
        d_j\int e^{t(\theta-\theta_j^\ell)}\,
        \lambda_j^\ell(d\theta) 
	+ a_0 \int e^{t(\theta-\theta_j^\ell)}\,
	(\theta-\theta_j^\ell)^2\,\nu_0^\ell(d\theta)}
	=1,
$$
which yields the desired contradiction.

\textbf{Step 4}.  Recall that $\lambda_j^\ell$ is supported on
$[\theta_j^{\ell-},\theta_j^{\ell+}]$ by construction.  We have
therefore established in the previous steps that the following hold:
$$
	d_j\langle u_\ell,\lambda_j(\mathbb{R}^d)u_\ell\rangle=
	\langle u_\ell,c_j u_\ell\rangle=
	a_0\,\nu_0^\ell([\theta_j^{\ell-},\infty[\mbox{})=0,\qquad
	a_i=0\mbox{ for }i>j.
$$
It is therefore easily seen that
$$
	0 = \lim_{t\to\infty}\frac{\Phi^{\ell,j}(t)}{t\,e^{t\theta_j^\ell}}
	= \langle b_j,u_\ell\rangle.
$$
Thus the proof is complete.
\end{proof}

We can now perform the induction by starting from (\ref{eq:laplace}) and
applying Proposition \ref{prop:finduction} repeatedly.  This yields
$d_j\langle u_\ell,\lambda_j(\mathbb{R}^d)u_\ell\rangle =
\langle u_\ell,c_j u_\ell\rangle = \langle b_j,u_\ell\rangle = 0$
for all $j=1,\ldots,q^\star$ and $\ell=1,\ldots,d$.  As $u_1,\ldots,u_d$
are linearly independent and $c_j\in M_+^d$, this implies 
that $b_j=0$, $c_j=0$ and $d_j=0$ for all $j=1,\ldots,q^\star$, so that
$$
	a_0\int e^{\ii\langle\theta,s\rangle}\,\nu_0(d\theta) +
	\sum_{i=1}^{q^\star}a_i\,e^{\ii\langle\theta_i^\star,s\rangle}
	= 0
$$
for all $s\in\mathbb{R}^d$ (this follows as above by Lemma \ref{lem:flem},
$h\equiv 0$, $F[f_0](s)\ne 0$ for $s$ in a neighborhood of the origin, and
using analyticity).  But
by the uniqueness of Fourier transforms, this implies that the signed 
measure $a_0\,\nu_0+\sum_{i=1}^{q^\star} a_i\,\delta_{\{\theta_i^\star\}}$ 
has no mass.  As $\nu_0$ is supported on $A_0$, this implies that 
$a_j=0$ for all $j=1,\ldots,q^\star$.
We have therefore shown that $a_i,b_i,c_i,d_i=0$ for all 
$i=1,\ldots,q^\star$.  But recall that 
$|a_0|+\sum_{i=1}^{q^\star}\{|a_i|+\|b_i\|+\mathrm{Tr}[c_i]+|d_i|\}=1$,
so that evidently $a_0=1$.

To complete the proof, it remains to note that
$$
	\int \frac{\ell(\eta^{n},\beta^{n},\rho^{n},\tau^{n},\nu^{n})}{
	N(\eta^{n},\beta^{n},\rho^{n},\tau^{n},\nu^{n})}\,f^\star d\mu
	= \sum_{i=0}^{q^\star} a_i^n
	\xrightarrow{n\to\infty} 1.
$$
But this is impossible, as 
$$
	\left\|
	\frac{\ell(\eta^{n},\beta^{n},\rho^{n},\tau^{n},\nu^{n})}{
	N(\eta^{n},\beta^{n},\rho^{n},\tau^{n},\nu^{n})}\right\|_1
	\xrightarrow{n\to\infty}0
$$
by construction.  Thus we have the desired contradiction.
\end{proof}

\subsection{Proof of Theorem \ref{thm:mainmixtures}}
\label{sec:proofmainmixtures}

The proof of Theorem \ref{thm:mainmixtures} consists of a sequence of
approximations, which we develop in the form of lemmas.  \emph{Throughout 
this section, we always presume that Assumption A holds.}

We begin by establishing the existence of an envelope function.

\begin{lemma}
\label{lem:senvelope}
Define $S = (H_0+H_1+H_2)\,d/c^\star$.  
Then $S\in L^4(f^\star d\mu)$, and 
$$
        \frac{|f/f^\star-1|}{\|f/f^\star-1\|_1} \le S
        \quad\mbox{for all }f\in\mathcal{M}.
$$
\end{lemma}

\begin{proof}
That $S\in L^4(f^\star d\mu)$ follows directly from Assumption A.
To proceed, let $f\in\mathcal{M}_q$, so that we can write $f=\sum_{i=1}^q 
\pi_i f_{\theta_i}$.  Then
$$
        \frac{f - f^\star}{f^\star} = 
	\sum_{j:\theta_j\in A_0}\pi_j \frac{f_{\theta_j}}{f^\star}+
        \sum_{i=1}^{q^\star}\Bigg\{
        \Bigg(
        \sum_{j:\theta_j\in A_i}\pi_j
        -\pi_i^\star\Bigg)\frac{f_{\theta_i^\star}}{f^\star}
        + \sum_{j:\theta_j\in A_i}
        \pi_j\,
        \frac{f_{\theta_j}-f_{\theta_i^\star}}{f^\star}
        \Bigg\}.
$$
Taylor expansion gives 
\begin{multline*}
	f_{\theta_j}(x)-f_{\theta_i^\star}(x)=(\theta_j-\theta_i^\star)^*
	D_1f_{\theta_i^\star}(x)+ \\
	\frac{1}{2}\int_{0}^{1}(\theta_j-\theta_i^\star)^*D_2f_{\theta_i^\star+u(\theta_j-\theta_i^\star)}(x)\,(\theta_j-\theta_i^\star)\, 2(1-u)\,du.
\end{multline*}
Using Assumption A, we find that
\begin{multline*}
        \left|\frac{f - f^\star}{f^\star}\right|\le 
        \Bigg[ \sum_{j:\theta_j\in A_0}\pi_j +\sum_{i=1}^{q^\star}\Bigg\{
        \Bigg|
        \sum_{j:\theta_j\in A_i}\pi_j
        -\pi_i^\star\Bigg| +
        \Bigg\|
        \sum_{j:\theta_j\in A_i}\pi_j(\theta_j-\theta_i^\star)
        \Bigg\|  
        \\ \mbox{} 
        + \frac{1}{2}
        \sum_{j:\theta_j\in A_i}\pi_j\|\theta_j-\theta_i^\star\|^2
        \Bigg\}\Bigg]
        \,(H_0+H_1+H_2)\,d.
\end{multline*}
On the other hand, Theorem \ref{thm:localgeom} gives
\begin{multline*}
        \left\|\frac{f - f^\star}{f^\star}\right\|_1\ge 
        c^\star
        \Bigg[ \sum_{j:\theta_j\in A_0}\pi_j +\sum_{i=1}^{q^\star}
        \Bigg\{
        \Bigg|
        \sum_{j:\theta_j\in A_i}\pi_j
        -\pi_i^\star\Bigg| 
        \\ \mbox{}
        +
        \Bigg\|
        \sum_{j:\theta_j\in A_i}\pi_j(\theta_j-\theta_i^\star)
        \Bigg\|  
        + \frac{1}{2}
        \sum_{j:\theta_j\in A_i}\pi_j\|\theta_j-\theta_i^\star\|^2
        \Bigg\}\Bigg].
\end{multline*} 
The proof follows directly.
\end{proof}

\begin{cor}
\label{cor:denvelope}
$|d|\le D$ for all $d\in\mathcal{D}$, where $D=2S\in L^4(f^\star d\mu)$.
\end{cor}

\begin{proof}
Using $\|f-f^\star\|_{\rm TV}\le 2h(f,f^\star)$ and
$|\sqrt{x}-1|\le |x-1|$, we find
$$
        |d_f| = \frac{|\sqrt{f/f^\star}-1|}{h(f,f^\star)} \le
        \frac{|f/f^\star-1|}{\frac{1}{2}\|f/f^\star-1\|_1} \le
        2S,
$$
where we have used Lemma \ref{lem:senvelope}.
\end{proof}

Next, we prove that the Hellinger normalized densities
$d_f$ can be approximated by chi-square normalized densities
for small $h(f,f^\star)$.

\begin{lemma}
\label{lem:chi2}
For any $f\in\mathcal{M}$, we have
$$
        \left|
        \frac{\sqrt{f/f^\star}-1}{h(f,f^\star)} -
        \frac{f/f^\star-1}{\sqrt{\chi^2(f||f^\star)}}
        \right| \le 
        \{4\|S\|_4^2 S+2S^2\}\,h(f,f^\star),
$$
where we have defined the chi-square divergence
$\chi^2(f||f^\star)=\|f/f^\star-1\|_2^2$.
\end{lemma}

\begin{proof}
Let us define the function $R$ as
$$
        \sqrt{\frac{f}{f^\star}}-1 = \frac{1}{2}\left\{
        \frac{f-f^\star}{f^\star}+R
        \right\}.
$$
Then we have
\begin{multline*}
        \frac{\sqrt{f/f^\star}-1}{h(f,f^\star)} -
        \frac{f/f^\star-1}{\sqrt{\chi^2(f||f^\star)}} =
        \frac{f/f^\star-1+R}{\|f/f^\star-1+R\|_2}-
        \frac{f/f^\star-1}{\|f/f^\star-1\|_2} = \mbox{}\\
        \frac{(f/f^\star-1+R)\{\|f/f^\star-1\|_2-\|f/f^\star-1+R\|_2\}
        +R\|f/f^\star-1+R\|_2
        }{\|f/f^\star-1+R\|_2\,\|f/f^\star-1\|_2},
\end{multline*}
so that by the reverse triangle inequality and Corollary 
\ref{cor:denvelope}
$$
        \left|
        \frac{\sqrt{f/f^\star}-1}{h(f,f^\star)} -
        \frac{f/f^\star-1}{\sqrt{\chi^2(f||f^\star)}}\right|
        \le
        \frac{2\|R\|_2S+|R|}{\|f/f^\star-1\|_2}.
$$
Now note that $R = -(\sqrt{f/f^\star}-1)^2\ge -(f/f^\star-1)^2$.
Therefore, by Lemma \ref{lem:senvelope},
$$
        |R| \le \left(\frac{f-f^\star}{f^\star}\right)^2
        \le S^2\left\|\frac{f-f^\star}{f^\star}\right\|_1^2
        \le S^2\left\|\frac{f-f^\star}{f^\star}\right\|_1
        \left\|\frac{f-f^\star}{f^\star}\right\|_2.
$$
The proof is easily completed using $\|f-f^\star\|_{\rm TV}\le
2h(f,f^\star)$.
\end{proof}

Finally, we need one further approximation step.

\begin{lemma}
\label{lem:deriv}
Let $q\in\mathbb{N}$ and $\alpha>0$.  Then for every $f\in\mathcal{M}_q$ 
such that $h(f,f^\star)\le\alpha$, it is possible to choose coefficients
$\eta_i\in\mathbb{R}$,
$\beta_{i}\in\mathbb{R}^{d}$, $\rho_i\in M_+^d$ for $i=1,\ldots,q^{\star}$,
and $\gamma_i\ge 0$, $\theta_i\in\Theta$ for $i=1,\ldots, q$,
such that $\sum_{i=1}^{q^\star}\mathrm{rank}[\rho_i]\le q\wedge dq^\star$,
$$
        \displaylines{
        \sum_{i=1}^{q^\star}|\eta_i|\le 
        \frac{1}{c^\star}+\frac{1}{\sqrt{c^\star\alpha}},
        \qquad
        \sum_{i=1}^{q^\star}\|\beta_i\|\le 
        \frac{1}{c^\star}+\frac{2T}{\sqrt{c^\star\alpha}},
        \cr
        \sum_{i=1}^{q^\star}\mathrm{Tr}[\rho_{i}]\le \frac{1}{c^\star},
        \qquad
        \sum_{j=1}^q|\gamma_j|\le \frac{1}{\sqrt{c^\star\alpha}\wedge c^\star},
        }
$$
and
$$
        \left|
        \frac{f/f^\star-1}{\sqrt{\chi^2(f||f^\star)}} -
        \ell
        \right|\le \frac{d^{3/2}\sqrt{2}}{3(c^\star)^{5/4}}\,
        \{\|H_3\|_2\, S + H_3\}\,\alpha^{1/4},
$$
where we have defined
$$
        \ell = \sum_{i=1}^{q^\star}\bigg\{
        \eta_i\frac{f_{\theta_i^\star}}{f^\star}
        +\beta_i^*\frac{D_1f_{\theta_i^\star}}{f^\star}
        +\mathrm{Tr}\bigg[\rho_i
	\frac{D_2f_{\theta_i^\star}}{f^\star}\bigg]
        \bigg\}
        + \sum_{j=1}^q\gamma_j\frac{f_{\theta_j}}{f^\star}.
$$
\end{lemma}

\begin{proof}
As $f\in\mathcal{M}_q$, we can write $f=\sum_{j=1}^q\pi_jf_{\theta_j}$.
Note that by Theorem \ref{thm:localgeom}
$$
        h(f,f^\star) 
        \ge \frac{c^\star}{4} \sum_{i=1}^{q^\star}
        \sum_{j:\theta_j\in A_i}\pi_j\|\theta_j-\theta_i^\star\|^2.
$$
Therefore, $h(f,f^\star)\le\alpha$ implies
$\pi_j\|\theta_j-\theta_i^\star\|^2\le 4\alpha/c^\star$ for
$\theta_j\in A_i$.  In particular, whenever $\theta_j\in A_i$, either 
$\pi_j\le 
2\sqrt{\alpha/c^\star}$ or $\|\theta_j-\theta_i^\star\|^2\le
2\sqrt{\alpha/c^\star}$.  Define
$$
        J=\bigcup_{i=1,\ldots,q^\star}\left\{j:\theta_j\in A_i,~
        \|\theta_j-\theta_i^\star\|^2\le 2\sqrt{\alpha/c^\star}\right\}.
$$
Taylor expansion gives 
$$
	f_{\theta_j}(x)-f_{\theta_i^\star}(x)=(\theta_j-\theta_i^\star)^*
	D_1f_{\theta_i^\star}(x)
	+\frac{1}{2}(\theta_j-\theta_i^\star)^*D_2f_{\theta_i^\star}(x)
	\,(\theta_j-\theta_i^\star)
	+ R_{ji}(x),
$$
where $|R_{ji}| \le \frac{1}{6}d^{3/2}\|\theta_j-\theta_i^\star\|^3H_3$.
We can therefore write
$$
        \frac{f-f^\star}{f^\star} = L +
        \sum_{i=1}^{q^\star}\sum_{j\in J:\theta_j\in A_i}
        \pi_j R_{ji},
$$
where we have defined
\begin{multline*}
        L = \sum_{i=1}^{q^\star}\Bigg\{
        \Bigg(\sum_{j\in J:\theta_j\in A_i}\pi_j-\pi_i^\star\Bigg)
        \frac{f_{\theta_i^\star}}{f^\star} +
        \sum_{j\in J:\theta_j\in A_i}\pi_j(\theta_j-\theta_i^\star)^*\,
        \frac{D_1f_{\theta_i^\star}}{f^\star} \\
        \mbox{} + \frac{1}{2}
        \sum_{j\in J:\theta_j\in A_i}\pi_j(\theta_j-\theta_i^\star)^*\,
        \frac{D_2f_{\theta_i^\star}}{f^\star}(\theta_j-\theta_i^\star)\Bigg\}
        + \sum_{j\not\in J}\pi_j\,\frac{f_{\theta_j}}{f^\star}.
\end{multline*}
Now note that
\begin{align*}
        \left|
        \frac{f/f^\star-1}{\sqrt{\chi^2(f||f^\star)}} -
        \frac{L}{\|L\|_2}
        \right| &\le
        \frac{|f/f^\star-1|}{\|f/f^\star-1\|_2}
        \frac{\|f/f^\star-1-L\|_2}{\|L\|_2}
        + \frac{|f/f^\star-1-L|}{\|L\|_2} \\
        &\le
        \frac{\|f/f^\star-1-L\|_2\,S+|f/f^\star-1-L|}{\|L\|_2},
\end{align*}
where we have used Lemma \ref{lem:senvelope}.
By Theorem \ref{thm:localgeom}, we obtain
$$
        \|L\|_2 \ge \|L\|_1 \ge
        \frac{c^\star}{2}\sum_{i=1}^{q^\star}
        \sum_{j\in J:\theta_j\in A_i}\pi_j\|\theta_j-\theta_i^\star\|^2.
$$
Therefore, we can estimate
$$
        \frac{|f/f^\star-1-L|}{\|L\|_2} \le
        \frac{d^{3/2}H_3}{3c^\star}
        \frac{\sum_{i=1}^{q^\star}\sum_{j\in J:\theta_j\in A_i}
        \pi_j\|\theta_j-\theta_i^\star\|^3}{
        \sum_{i=1}^{q^\star}
        \sum_{j\in J:\theta_j\in A_i}\pi_j\|\theta_j-\theta_i^\star\|^2} \le
        \left(\frac{4\alpha}{c^\star}\right)^{1/4}
        \frac{d^{3/2}H_3}{3c^\star}
$$
where we have used the definition of $J$.  Setting $\ell=L/\|L\|_2$,
we obtain
$$
        \left|
        \frac{f/f^\star-1}{\sqrt{\chi^2(f||f^\star)}} -
        \ell
        \right| \le
        \frac{d^{3/2}\sqrt{2}}{3(c^\star)^{5/4}}\,
        \{\|H_3\|_2\, S + H_3\}\,\alpha^{1/4}.
$$
It remains to show that for our choice of $\ell=L/\|L\|_2$, the coefficients
$\eta,\beta,\rho,\gamma$ in the statement of the lemma satisfy the 
desired bounds.  These coefficients are
$$
        \displaylines{
        \eta_i=\frac{1}{\|L\|_2}\Bigg(
        \sum_{j\in J:\theta_j\in A_i}\pi_j-\pi_i^\star\Bigg),\qquad
        \beta_i=\frac{1}{\|L\|_2}\sum_{j\in J:\theta_j\in A_i}
                \pi_j(\theta_j-\theta_i^\star),
        \cr
        \rho_{i}=\frac{1}{2\|L\|_2}
        \sum_{j\in J:\theta_j\in A_i}
                \pi_j(\theta_j-\theta_i^\star)
		(\theta_j-\theta_i^\star)^*,
	\qquad
        \gamma_j=\frac{\pi_j\one_{j\not\in J}}{\|L\|_2}.
        }
$$
Clearly $\mathrm{rank}[\rho_i]\le \#\{j:\theta_j\in A_i\}\wedge d$, so
$\sum_{i=1}^{q^\star}\mathrm{rank}[\rho_i]\le q\wedge dq^\star$.
Moreover,
\begin{multline*}
        \|L\|_2 \ge c^\star\Bigg[\sum_{j:\theta_j\in A_0}\pi_j +
	\sum_{i=1}^{q^\star}
        \Bigg\{
        \Bigg|\sum_{j:\theta_j\in A_i}\pi_j-\pi_i^\star\Bigg| \\
        +\Bigg\|
        \sum_{j:\theta_j\in A_i}\pi_j(\theta_j-\theta_i^\star)\Bigg\|
        + \frac{1}{2}
        \sum_{j:\theta_j\in A_i}\pi_j\|\theta_j-\theta_i^\star\|^2\Bigg\}
	\Bigg]
\end{multline*}
by Theorem \ref{thm:localgeom}.  It follows that 
$\sum_{i=1}^{q^\star}\mathrm{Tr}[\rho_{i}]\le 1/c^\star$.
Now note that for $j\not\in J$ such that $\theta_j\in A_i$, we have 
$\|\theta_j-\theta_i^\star\|^2>2\sqrt{\alpha/c^\star}$ by construction.
Therefore
$$
        \|L\|_2 
        \ge c^\star\Bigg[
	\sum_{j\not\in J:\theta_j\in A_0}\pi_j +
	\frac{1}{2}\sum_{i=1}^{q^\star}
        \sum_{j\not\in J:\theta_j\in A_i}
        \pi_j\|\theta_j-\theta_i^\star\|^2 
	\Bigg]
        \ge ( \sqrt{c^\star\alpha} \wedge c^\star )\,\sum_{j\not\in J}\pi_j.
$$
It follows that $\sum_{j=1}^q|\gamma_j|\le 1/ ( \sqrt{c^\star\alpha} \wedge c^\star )$.
Next, we note that
$$
        \sum_{i=1}^{q^\star}
        \Bigg|\sum_{j\in J:\theta_j\in A_i}\pi_j-\pi_i^\star\Bigg|
        \le
        \sum_{i=1}^{q^\star}
        \Bigg|\sum_{j:\theta_j\in A_i}\pi_j-\pi_i^\star\Bigg| +
        \sum_{j\not\in J:\theta_j\not\in A_0}\pi_j.
$$
Therefore $\sum_{i=1}^{q^\star}|\eta_i|\le 
1/c^\star+1/\sqrt{c^\star\alpha}$.  Finally, note that
$$
        \sum_{i=1}^{q^\star}
        \Bigg\|\sum_{j\in J:\theta_j\in A_i}
                \pi_j(\theta_j-\theta_i^\star)\Bigg\| \le
        \sum_{i=1}^{q^\star}
        \Bigg\|\sum_{j:\theta_j\in A_i}
                \pi_j(\theta_j-\theta_i^\star)\Bigg\| +
        2T\sum_{j\not\in J:\theta_j\not\in A_0}\pi_j.
$$
Therefore $\sum_{i=1}^{q^\star}\|\beta_i\|\le
1/c^\star+2T/\sqrt{c^\star\alpha}$.  The proof is complete.
\end{proof}

We can now complete the proof of Theorem \ref{thm:mainmixtures}.

\begin{proof}[Proof of Theorem \ref{thm:mainmixtures}]
Let $\alpha>0$ be a constant to be chosen later on, and
$$
        \mathcal{D}_{q,\alpha}=\{d_f:f\in\mathcal{M}_q,~
        f\ne f^\star,~
        h(f,f^\star)\le\alpha\}.
$$
Then clearly
$$
        N_{[]}(\mathcal{D}_q,\delta) \le
        N_{[]}(\mathcal{D}_{q,\alpha},\delta) +
        N_{[]}(\mathcal{D}_q\backslash\mathcal{D}_{q,\alpha},\delta).
$$
We will estimate each term separately.

\textbf{Step 1} (\emph{the first term}). 
Define
$$
	\mathbb{M}_{q}=
	\{(m_{1},\ldots,m_{q^\star})\in
	\mathbb{Z}_+^{q^\star}:m_{1}+\cdots+m_{q^\star}=q
	\wedge dq^\star\}.
$$
For every $m\in\mathbb{M}_q$, we define the family of functions
\begin{multline*}
\mathcal{L}_{q,m,\alpha}=       \Bigg\{
        \sum_{i=1}^{q^\star}\bigg\{
        \eta_i\frac{f_{\theta_i^\star}}{f^\star}
        +\beta_i^*\frac{D_1f_{\theta_i^\star}}{f^\star}
        +\sum_{j=1}^{m_{i}}\rho_{ij}^*\frac{D_2f_{\theta_i^\star}}{f^\star}\rho_{ij}
        \bigg\}
        + \sum_{j=1}^{q}\gamma_j\frac{f_{\theta_j}}{f^{\star}}:\mbox{}\\
        (\eta,\beta,\rho,\gamma,\theta)\in\mathfrak{I}_{q,m,\alpha}\Bigg\},
\end{multline*}
where
\begin{multline*}
        \mathfrak{I}_{q,m,\alpha}=\Bigg\{(\eta,\beta,\rho,\gamma,\theta)\in
        \mathbb{R}^{q^\star}\times (\mathbb{R}^d)^{q^\star}\times
        (\mathbb{R}^d)^{m_{1}}\times\cdots\times (\mathbb{R}^d)^{m_{q^\star}}
	\times \mathbb{R}^q\times\Theta^q: \mbox{}\\
        \sum_{i=1}^{q^\star}|\eta_i|\le 
        \frac{1}{c^\star}+\frac{1}{\sqrt{c^\star\alpha}},
        \qquad
        \sum_{i=1}^{q^\star}\|\beta_i\|\le 
        \frac{1}{c^\star}+\frac{2T}{\sqrt{c^\star\alpha}},
        \mbox{}\\
        \sum_{i=1}^{q^\star}\sum_{j=1}^{m_{i}}\|\rho_{ij}\|^2\le \frac{1}{c^\star},
        \qquad
        \sum_{j=1}^q|\gamma_j|\le \frac{1}{\sqrt{c^\star\alpha}\wedge c^\star}
        \Bigg\}.
\end{multline*}
Define the family of functions
$$
\mathcal{L}_{q,\alpha}= \bigcup_{m\in\mathbb{M}_{q}}\mathcal{L}_{q,m,\alpha}
$$
From Lemmas \ref{lem:chi2} and \ref{lem:deriv}, we find that for
any function $d\in\mathcal{D}_{q,\alpha}$, there exists a function 
$\ell\in \mathcal{L}_{q,\alpha}$ such that (here we use that
$h(f,f^\star)\le\sqrt{2}$ for any $f$)
$$
        |d-\ell| \le 
        \{4\|S\|_4^2 S+2S^2\}\,(\alpha\wedge\sqrt{2})
        + \frac{d^{3/2}\sqrt{2}}{3(c^\star)^{5/4}}\,
        \{\|H_3\|_2\, S + H_3\}\,\alpha^{1/4}.
$$
Using $\alpha\wedge\sqrt{2}\le 2^{3/8}\alpha^{1/4}$ for all $\alpha>0$, we 
can estimate
$$
        |d-\ell| \le \alpha^{1/4}\,U,\qquad
        U=\left( \frac{1+\|H_3\|_2}{(c^\star)^{5/4}}+
        8\|S\|_4^2 + 4\right)d^{3/2}\,\{S+S^2+H_3\},
$$
where $U\in L^2(f^\star d\mu)$ by Assumption A.  
Now note that if $m_1\le\ell\le m_2$ for some functions $m_1,m_2$ with 
$\|m_2-m_1\|_2\le\varepsilon$, then $m_1-\alpha^{1/4}\,U\le d\le 
m_2+\alpha^{1/4}\,U$ with $\|(m_2+\alpha^{1/4}\,U)
-(m_1-\alpha^{1/4}\,U)\|_2\le \varepsilon+2\alpha^{1/4}\|U\|_2$.
Therefore
$$
        N_{[]}(\mathcal{D}_{q,\alpha},
        \varepsilon+2\alpha^{1/4}\|U\|_2) \le
        N_{[]}(\mathcal{L}_{q,\alpha},\varepsilon) \le
        \sum_{m\in\mathbb{M}_{q}}N_{[]}(\mathcal{L}_{q,m,\alpha},\varepsilon)
        \quad
        \mbox{for }\varepsilon>0.
$$
Of course, we will ultimately choose $\varepsilon,\alpha$ such that
$\varepsilon+2\alpha^{1/4}\|U\|_2=\delta$.

We proceed to estimate the bracketing number 
$N_{[]}(\mathcal{L}_{q,m,\alpha},\varepsilon)$. To this end, let 
$\ell,\ell'\in\mathcal{L}_{q,m,\alpha}$, where $\ell$ is defined by the 
parameters $(\eta,\beta,\rho,\gamma,\theta)\in\mathfrak{I}_{q,m,\alpha}$ and 
$\ell'$ is defined by the parameters 
$(\eta',\beta',\rho',\gamma',\theta')\in\mathfrak{I}_{q,m,\alpha}$.  Note that
$$
	\sum_{i=1}^{q^\star}\sum_{j=1}^{m_i}
	\Bigg|
	\rho_{ij}^*\frac{D_2f_{\theta_i^\star}}{f^\star}\rho_{ij}-
	(\rho_{ij}')^*\frac{D_2f_{\theta_i^\star}}{f^\star}\rho_{ij}'
	\Bigg|
	\le
	\frac{2d}{\sqrt{c^\star}}\,H_2
	\sum_{i=1}^{q^\star}\sum_{j=1}^{m_i}
	\,\|\rho_{ij}-\rho_{ij}'\|.
$$
We can therefore estimate
\begin{multline*}
        |\ell-\ell'| \le
	H_0\sum_{i=1}^{q^\star}|\eta_i-\eta_i'| +
	H_1\sqrt{d}\sum_{i=1}^{q^\star}\|\beta_i-\beta_i'\| + 
	H_0\sum_{j=1}^q|\gamma_j-\gamma_j'| +
	\\
	\frac{\sqrt{d}}{\sqrt{c^\star\alpha}\wedge c^\star}
	\,H_1 \max_{j=1,\ldots,q}\|\theta_j-\theta_j'\| +
	\frac{2d\sqrt{dq^\star}}{\sqrt{c^\star}}\,H_2\,
	\Bigg[
	\sum_{i=1}^{q^\star}\sum_{j=1}^{m_i}
	\,\|\rho_{ij}-\rho_{ij}'\|^2\Bigg]^{1/2}.
\end{multline*}
where we have used that $|f_{\theta}-f_{\theta'}|/f^\star \le
\|\theta-\theta'\|\,H_1 \sqrt{d}$ by Taylor expansion.
Therefore, writing $V=(H_0+H_1+H_2)\,d\sqrt{dq^\star}$, we have
$$
        |\ell-\ell'| \le V\,
        \tnorm{(\eta,\beta,\rho,\gamma,\theta)-
        (\eta',\beta',\rho',\gamma',\theta')}_{q,m,\alpha},
$$
where $\tnorm{\cdot}_{q,m,\alpha}$ is the norm on
$\mathbb{R}^{(1+d)q^\star+ d(q\wedge dq^\star)+(1+d)q}$ defined by
\begin{multline*}
        \tnorm{(\eta,\beta,\rho,\gamma,\theta)}_{q,m,\alpha} = 
        \sum_{i=1}^{q^\star}|\eta_i|
	+ \sum_{i=1}^{q^\star}\|\beta_i\|
        + \sum_{j=1}^q|\gamma_j| \\
	+
        \frac{1}{\sqrt{c^\star\alpha}\wedge c^\star}
	\max_{j=1,\ldots,q}\|\theta_j\|
	+
	\frac{2}{\sqrt{c^\star}}\,
	\Bigg[
	\sum_{i=1}^{q^\star}\sum_{j=1}^{m_i}
	\,\|\rho_{ij}\|^2\Bigg]^{1/2}.
\end{multline*}
Note that if $\tnorm{(\eta,\beta,\rho,\gamma,\theta)-
(\eta',\beta',\rho',\gamma',\theta')}_{q,m,\alpha}\le \varepsilon'$,
then we obtain a bracket $\ell'-\varepsilon'V\le \ell \le 
\ell'+\varepsilon'V$ of size
$\|(\ell'+\varepsilon'V)-(\ell'-\varepsilon'V)\|_2=
2\varepsilon'\|V\|_2$.  Thus 
$$
        N_{[]}(\mathcal{L}_{q,m,\alpha},\varepsilon)
        \le N(\mathfrak{I}_{q,m,\alpha},\tnorm{\cdot}_{q,m,\alpha},
        \varepsilon/2\|V\|_2)\quad
        \mbox{for }\varepsilon>0,
$$
where $N(\mathfrak{I}_{q,m,\alpha},\tnorm{\cdot}_{q,m,\alpha},\varepsilon')$
denotes the covering number of $\mathfrak{I}_{q,m,\alpha}$ with respect to the
$\tnorm{\cdot}_{q,m,\alpha}$-norm.
But note that, by construction, $\mathfrak{I}_{q,m,\alpha}$ is included
in a $\tnorm{\cdot}_{q,m,\alpha}$-ball of radius not exceeding
$(6+3T)/(\sqrt{c^\star\alpha}\wedge c^\star)$.  Therefore, using the
standard fact that the covering number of the $r$-ball 
$B(r)=\{x\in B:\tnorm{x}\le r\}$ in any $n$-dimensional normed space 
$(B,\tnorm{\cdot})$ satisfies $N(B(r),\tnorm{\cdot},\varepsilon)\le 
(\frac{2r+\varepsilon}{\varepsilon})^n$, we obtain
$$
        N_{[]}(\mathcal{L}_{q,m,\alpha},\varepsilon) \le
	\left(\frac{4\|V\|_2(6+3T)/(\sqrt{c^\star\alpha}\wedge c^\star)
	+\varepsilon}{\varepsilon}
	\right)^{(1+d)q^\star+ d(q\wedge dq^\star)+(1+d)q}.
$$
In particular, if $\varepsilon\le 1$ and $\alpha\le c^\star$, then
$$
        N_{[]}(\mathcal{L}_{q,m,\alpha},\varepsilon) \le
	\left(\frac{(24+12T)\|V\|_2/\sqrt{c^\star}
	+\sqrt{c^\star}}{\varepsilon\sqrt{\alpha}}
	\right)^{3(d+1)q}.
$$
Finally, note that the cardinality of $\mathbb{M}_q$ can be
estimated as
$$
	\#\mathbb{M}_q = 
	{q^\star + q\wedge dq^\star -1 \choose q\wedge dq^\star}
	\le
	2^{2q},
$$
where we have used that ${n\choose k}\le 2^n$ and
$q\ge q^\star$.  We therefore obtain
\begin{align*}
        N_{[]}(\mathcal{D}_{q,\alpha},\delta) &\le
        \sum_{m\in\mathbb{M}_{q}}
	N_{[]}(\mathcal{L}_{q,m,\alpha},
	\delta-2\alpha^{1/4}\|U\|_2) \\
	&\le
	\left(\frac{24(2+T)\|V\|_2/\sqrt{c^\star}
	+\sqrt{c^\star}}{(\delta-2\alpha^{1/4}\|U\|_2)\sqrt{\alpha}}
	\right)^{3(d+1)q}
\end{align*}
whenever $\delta\le 1$ and $\alpha\le (\delta/2\|U\|_2)^4\wedge c^\star$.

\textbf{Step 2} (\emph{the second term}).
  For $f,f'\in\mathcal{M}_q$
with $h(f,f^\star)>\alpha$ and $h(f',f^\star)>\alpha$,
\begin{align*}
        |d_f-d_{f'}| &=
	\frac{|
	(\sqrt{f/f^\star}-1)(h(f',f^\star)-h(f,f^\star))+
	(\sqrt{f/f^\star}-\sqrt{f'/f^\star})h(f,f^\star)|}{h(f,f^\star)h(f',f^\star)}
        \\ &\le 
	\frac{
	|\sqrt{f/f^\star}-1|}{h(f,f^\star)}
	\frac{\|\sqrt{f'/f^\star}-\sqrt{f/f^\star}\|_2}
	{h(f',f^\star)}
	+
	\frac{|\sqrt{f/f^\star}-\sqrt{f'/f^\star}|}{h(f',f^\star)}
        \\ &\le 
	\|\sqrt{f'/f^\star}-\sqrt{f/f^\star}\|_2\,
	\frac{2S}{\alpha}
	+
	\frac{|\sqrt{f/f^\star}-\sqrt{f'/f^\star}|}{\alpha},
\end{align*}
where we have used Corollary \ref{cor:denvelope}.
Now note that 
$$
        \big|\sqrt{a}-\sqrt{b}\big|^2 \le 
        \big|\sqrt{a}-\sqrt{b}\big|\left(\sqrt{a}+\sqrt{b}\right)
        =|a-b|
$$
for any $a,b\ge 0$.  We can therefore estimate
$$
        |d_f-d_{f'}| \le
        \frac{
        \|(f-f')/f^\star\|_1^{1/2}2S +
        |(f-f')/f^\star|^{1/2}}{\alpha},
$$
Now note that if we write 
$f=\sum_{i=1}^q\pi_if_{\theta_i}$ and 
$f'=\sum_{i=1}^q\pi_i'f_{\theta_i'}$, then 
$$
        \left|\frac{f-f'}{f^\star}\right| \le
        H_0\sum_{i=1}^q|\pi_i-\pi_i'| + 
        H_1\sqrt{d}\max_{i=1,\ldots,q}\|\theta_i-\theta_i'\|.
$$
Defining
$$
        W = \|H_0+H_1\sqrt{d}\|_1^{1/2}2S +
        (H_0+H_1\sqrt{d})^{1/2},
$$
we obtain
$$
        |d_f-d_{f'}| \le
        \frac{W}{\alpha}\,\tnorm{(\pi,\theta)-(\pi',\theta')}_q^{1/2},
        \qquad
        \tnorm{(\pi,\theta)}_q=\sum_{i=1}^q|\pi_i| +
        \max_{i=1,\ldots,q}\|\theta_i\|
$$
(clearly $\tnorm{\cdot}_q$ defines a norm on $\mathbb{R}^{(d+1)q}$).  
Now note that if $\tnorm{(\pi,\theta)-(\pi',\theta')}_q\le\varepsilon$, 
then we obtain a bracket $d_{f'}-\varepsilon^{1/2}W/\alpha\le
d_f\le d_{f'}+\varepsilon^{1/2}W/\alpha$ of size
$\|(d_{f'}+\varepsilon^{1/2}W/\alpha)-(d_{f'}-\varepsilon^{1/2}W/\alpha)\|_2
= 2\varepsilon^{1/2}\|W\|_2/\alpha$.  Therefore
$$
        N_{[]}(\mathcal{D}_q\backslash\mathcal{D}_{q,\alpha},\delta)
        \le
        N(\Delta_q\times\Theta^q,
        \tnorm{\cdot}_q,\alpha^2\delta^2/4\|W\|_2^2),
$$
where we have defined the simplex 
$\Delta_q=\{\pi\in\mathbb{R}^q_+:\sum_{i=1}^q\pi_i=1\}$.
We can now estimate the quantity on the right hand side of this expression 
as before, giving
$$
        N_{[]}(\mathcal{D}_q\backslash\mathcal{D}_{q,\alpha},\delta)
        \le
	\left(
	\frac{8(1+T)\|W\|_2^2 +(c^\star)^4}{\alpha^2\delta^2}
	\right)^{(d+1)q}
$$
for $\delta\le 1$ and $\alpha\le c^\star$.

\textbf{End of proof.}  Choose $\alpha=(\delta/4\|U\|_2)^4$.  Collecting 
the various estimates above, we find that for 
$\delta\le 1\wedge 4(c^\star)^{1/4}$ (as $\|U\|_2\ge\|S\|_1\ge 1$
by Lemma \ref{lem:senvelope})
\begin{align*}
        N_{[]}(\mathcal{D}_{q},\delta) 
        &\le
	\left(\frac{768(2+T)\|U\|_2^2\|V\|_2/\sqrt{c^\star}
	+32\|U\|_2^2\sqrt{c^\star}}{\delta^3}
	\right)^{3(d+1)q} \\
	&\qquad\quad\mbox{}+
	\left(
	\frac{4^{10}(1+T)\|U\|_2^{8}\|W\|_2^2 
	+ 4^{8}\|U\|_2^{8}(c^\star)^4}{\delta^{10}}
	\right)^{(d+1)q} \\
	&\le
	\left(
	\frac{
	c_0^\star\,(T\vee 1)^{1/3}\,
	(\|U\|_2\vee\|V\|_2\vee\|W\|_2)
	}{\delta}
	\right)^{10(d+1)q} 
\end{align*}
where $c_0^\star$ is a constant depends only on $c^\star$.  It follows that
$$
        N_{[]}(\mathcal{D}_{q},\delta) \le
	\left(
	\frac{
	C^\star (T\vee 1)^{1/3} 
	(\|H_0\|_4^4\vee\|H_1\|_4^4\vee\|H_2\|_4^4\vee\|H_3\|_2^2)
	}{\delta}
	\right)^{10(d+1)q} 
$$
for all $\delta\le\delta^\star$, where $C^\star$ and $\delta^\star$
are constants that depend only on $c^\star$, $d$, and $q^\star$.
This establishes the estimate given in the statement of the Theorem.
The proof of the second half of the Theorem follows from 
Corollary \ref{cor:denvelope} and $\|H_0\|_4\ge 1$.
\end{proof}

\textbf{Acknowledgments.}
The authors thank Jean Bretagnolle for providing an 
enlightening counterexample that guided some of our proofs, and
the anonymous referee for very helpful comments that have improved
the presentation.

\bibliographystyle{amsplain}
\bibliography{refbis}

\end{document}